\documentclass[12pt,oneside]{amsart}
\usepackage{amsmath, amscd}
\usepackage{hyperref}
\usepackage{graphicx}
\usepackage[english]{babel}

\usepackage{amsthm}

\newtheorem{thm}{Theorem}[section]
\newtheorem{cor}[thm]{Corollary}
\newtheorem{lem}[thm]{Lemma}
\newtheorem{prop}[thm]{Proposition}
\newtheorem{exmp}[thm]{Example}
\newtheorem{claim}[thm]{Claim}

\theoremstyle{definition}
\newtheorem{defn}{Definition}[section]
\theoremstyle{remark}
\newtheorem{rem}{Remark}[section]

      \newcommand{\R}{{\mathbb{R}}}
      \newcommand{\Z}{{\mathbb Z}}

      \newcommand{\D}{\mathrm{Diff}}
      \newcommand{\mcg}{\mathcal{M}^{\infty}}
       \newcommand{\pmcg}{\mathcal{PM}^{\infty}}
       \newcommand{\mcgi}{\mathcal{M}_0^{\infty}}
       
      \newcommand{\mcgt}{\mathcal{M}^{0}}
      \newcommand{\pmcgt}{\mathcal{PM}^{0}}      
      \newcommand{\homeo}{\text{Homeo}}
      
        \newcommand{\s}{\mathbb{S}}
       
      \newcommand{\e}{\epsilon}
      
           \newcommand{\Id}{\text{Id}}
           
               \newcommand{\atr}{\text{Att}}
          \newcommand{\rep}{\text{Rep}}
          \newcommand{\per}{\text{Per}_0}

      \makeatletter
      \def\@setcopyright{}
      \def\serieslogo@{}
      \makeatother
   
\begin{document}

  \author[Sebastian Hurtado]{Sebastian Hurtado, Emmanuel Militon}

   \title[Distortion and Tits alternative in $\mcg(S,K)$]{Distortion and Tits alternative in smooth mapping class groups}

   \date{\today}

\begin{abstract}
In this article, we study the smooth mapping class group of a surface $S$ relative to a given Cantor set, that is the group of isotopy classes of orientation-preserving smooth diffeomorphisms of $S$ which preserve this Cantor set. When the Cantor set is the standard ternary Cantor set, we prove that this group does not contain any distorted element. Moreover, we prove a weak Tits alternative for these groups.
\end{abstract}

\maketitle

\begin{section}{Introduction}

\begin{defn} Let $S$ be a surface of finite type and let $K$ be a closed subset contained in $S$. Let $\D(S,K)$ be the group of orientation-preserving $C^{\infty}$-diffeomorphisms of $S$ that leave $K$ invariant (\emph{i.e.} $f(K) = K$) and let $\D_0(S,K)$ be the identity component of $\D(S,K)$. We define  the ``smooth" mapping class group $\mcg(S,K)$  of $S$ relative to $K$ as the quotient: $$\mcg(S,K) = \D(S,K)/\D_0(S,K)$$

\end{defn}

If $K$ is a finite set of points, $\mcg(S,K)$ coincides with the braid group with $|K|$ points in $S$.\\

The groups  $\mcg(S,K)$ appear very naturally when studying group actions on surfaces, as given a smooth group action on $S$ preserves a non-trivial closed set $K$, one obtains a group homomorphism into $\mcg(S,K)$. These groups were first studied by Funar and Neretin in \cite{FN}. The aim of this paper is to contribute to the study of $\mcg(S,K)$ for a Cantor set $K \subset S$. Our results are aiming towards the understanding of two basic properties of these groups, namely, the distortion of the cyclic subgroups and the Tits alternative.\\

A recent result related to mapping class groups of infinite type that deserves to be mentioned (but which we will not make use of here) is J.Bavard's proof that the mapping class group of $\R^2$ relative to a Cantor set $K$ acts faithfully on a Gromov hyperbolic space. This hyperbolic space is similar to the curve complex for finite mapping class groups and  was suggested by Calegari (see \cite{Ba} and \cite{Ca}). Possible lines of research and developments after Bavard's article are suggested by Calegari in his blog (see \cite{Ca}). These developments partly inspired this work.\\

We now introduce some notation in order to state our results and explain the ideas involved in the proofs. Denote by $\pmcg(S, K) \subset \mcg(S, K)$ (the ``pure'' smooth mapping class group of $K$ in $S$) the subgroup of the mapping class group $\mcg(S,K)$ consisting of the elements which fix $K$ pointwise.\\

We also define the group $\mathfrak{diff}_{S}(K)$ as the group of homeomorphisms of $K$ which are induced by orientation-preserving diffeomorphisms of $S$ preserving $K$. In other words, a homeomorphism $f : K \to K$ belongs to $\mathfrak{diff}_{S}(K)$ if there exists $\bar{f} \in \D(S)$ such that $\bar{f}|_{K} = f$. There is a natural exact sequence of groups given by: 

\begin{equation}\label{exseq}
\pmcg(S, K) \to  \mcg(S, K) \to \mathfrak{diff}_{S}(K).
\end{equation} 

The exact sequence~\eqref{exseq}  was studied by Funar and Neretin in \cite{FN}, where it is proven that  $\pmcg(S,K)$  is always a countable group (this follows from Lemma \ref{handel}) and where for certain ``affine'' Cantor sets $K$, the group $\mathfrak{diff}_{S}(K)$ is shown to be countable and to have a ``Thompson group" kind of structure.\\

We can now proceed to state our results. We begin with our results about distortion in $\mcg(S,K)$.

\begin{subsection}{Distortion} 

We recall the concept of distortion which comes from geometric group theory and is due to Gromov (see \cite{Gro}).\\

\begin{defn} Let $G$ be a group and let $\mathcal{G} \subset G$ be a finite set which is symmetric (\emph{i.e.} $\mathcal{G}= \mathcal{G}^{-1}$). For any element $f \in G$ contained in the group generated by $\mathcal{G}$, the word length $l_{\mathcal{G}}(f)$ of $f$ is defined as the minimal integer $n$ such that $f$ can be expressed as a product $$f = s_1s_2...s_n$$ where each $s_i \in \mathcal{G}$. An element $f$ of a group $G$ is called distorted if it has infinite order and there exists a finite symmetric set $\mathcal{G} \subset G$ such that
\begin{enumerate}
\item $f \in \langle \mathcal{G} \rangle$.
\item $\underset{n \to \infty}{\lim} \frac{l_{\mathcal{G}}(f^n)}{n} = 0$.
\end{enumerate}
\end{defn}

By a theorem of Farb, Lubotzky and Minsky, the mapping class group of a surface of finite type does not have distorted elements (see \cite{FLM}). We believe that such a result should extend to the groups $\mcg(S,K)$ for every Cantor set $K$.\\ 

Our first result is the following. Recall that a Cantor set is a nonempty totally disconnected topological space such that any point of $K$ is an accumulation point.

\begin{thm}\label{m1} Let $S$ be a closed surface and $K$ be a closed subset of $S$ which is a Cantor set. The elements of $\pmcg(S, K)$ are undistorted in $\mcg(S,K)$.
\end{thm}

If we change the $C^{\infty}$ regularity into a $C^{r}$ regularity, for some $r \geq 1$, this theorem is still true and can be proved in the same way. It is important to point out that any element in $\pmcg(S, K)$ can be thought as a mapping class of a surface of finite type (see Corollary \ref{local}) and is therefore much easier to deal with compared to other elements of $\mcg(S,K)$. In fact, this group is a limit of finite-type surface mapping class groups.\\

Theorem \ref{m1} is proven using some of the techniques developed by Franks and Handel in \cite{FH} to show that there are no distorted element in the groups of area-preserving diffeomorphisms of surfaces and techniques used by Alibegovic in \cite{Ali} to prove the same thing in the outer automorphism group of a free group.\\

Denote by $\mathcal{M}^{0}(S, K)$ the quotient of the group of orientation-preserving homeomorphisms of $S$ which preserve $K$ by the subgroup consisting of homeomorphisms of $S$ which are isotopic to the identity relative to $K$. It is worth pointing out that if the smoothness assumption is dropped and if $K$ is a Cantor set, one can construct distorted elements in $\mathcal{M}^{0}(S, K)$. One can even construct elements which fix the set $K$ pointwise and which are distorted in the subgroup $\mathcal{M}^{0}_{0}(S, K)$ of $\mathcal{M}^{0}(S, K)$ consisting of homeomorphisms isotopic to the identity in $S$. In particular, Theorem \ref{m2} below implies that the group $\mathcal{M}^{0}(S, K)$ is not isomorphic to $\mcg(S,K)$: the behavior of these groups is different from the behavior of classical mapping class groups (\emph{i.e.} when $K$ is a finite set).\\

 In the smooth case, we have not been able so far to produce any distorted element for our groups $\mcg(S,K)$ (even when $S = \s^1$ is the circle and $K \subset \s^1$).\\

We will then focus on one of the simplest examples of Cantor sets in surfaces: the standard ternary Cantor set $C_{\lambda}$ contained in an embedded segment $l \subset S$ (see Section \ref{scs} for a precise definition of $C_{\lambda}$). It is shown in \cite{FN} that the group $ \mathfrak{diff}_{S}(C_{\lambda})$ consists of piecewise affine homeomorphisms of $C_{\lambda}$ and therefore $\mathfrak{diff}_{S}(K)$ is very``similar" to Thompson's group $V_2$, see Lemma \ref{v2v2}.\\

Using the fact that there are no distorted elements in Thompson's groups $V_n$ (see \cite{B}) and the exact sequence~\eqref{exseq}, we are then able to show the following:

\begin{thm}\label{m2}
Let $S$ be a closed surface and $0< \lambda < 1/2$. Then, there are no distorted elements in the group $\mcg(S, C_{\lambda})$, where $C_{\lambda}$ is an embedding of the standard ternary Cantor set with parameter $\lambda$ in $S$.
\end{thm}

Here also , the theorem holds in case of a $C^{r}$ regularity, for $r \geq 1$.

\end{subsection}

\begin{subsection}{Tits alternative} 

By a theorem by McCarthy (see \cite{Mac}), mapping class groups of finite type satisfy the Tits alternative, i.e., any subgroup $\Gamma \subset \mcg(S)$ either contains a free subgroup on two generators or is virtually solvable. In \cite{M}, Margulis proved that the group $\text{Homeo}(\s^1)$ of homeomorphisms of the circle satisfies a similar alternative. More precisely, he proved that a group $\Gamma \subset \text{Homeo}(\s^1)$ either preserves a measure on $\s^1$ or contains a free subgroup on two generators, see \cite{N}. \\

Ghys asked whether the same statement holds for $\D(S)$ for a surface $S$ (see \cite{G}). We believe that some kind of similar statement should hold for our groups $\mcg(S, K)$. Here, we obtain the following result in the case where the Cantor set $K$ is the standard ternary Cantor set $C_{\lambda}$.

\begin{thm}\label{m3}
Let $\Gamma$ be a finitely generated subgroup of $\mcg(S, C_{\lambda})$, then one of the following holds:

\begin{enumerate}

\item $\Gamma$ contains a free subgroup on two generators.
\item $\Gamma$ has a finite orbit, i.e. there exists $p \in C_{\lambda}$ such that the set $\Gamma(p) := \{ g(p) \: | g \in \Gamma\}$ is finite.

\end{enumerate}

\end{thm}

This theorem also holds in case of a $C^{r}$ regularity, for $r \geq 1$.

Note also the following immediate corollary of Theorem \ref{m3}. This corollary is similar to Margulis's theorem on the group of homeomorphisms of the circle.

\begin{cor}
Let $\Gamma$ be a subgroup of $\mcg(S, C_{\lambda})$, then one of the following holds:

\begin{enumerate}

\item $\Gamma$ contains a free subgroup on two generators.
\item The action of $\Gamma$ on the Cantor set $C_{\lambda}$ has an invariant probability measure.
\end{enumerate}
\end{cor}

To obtain the corollary from the theorem, observe that, as the group $\mcg(S, C_{\lambda})$ is countable, such a subgroup $\Gamma$ can be written as the union of a sequence $(\Gamma_{p})_{p \in \mathbb{N}}$ of finitely generated subgroups with $\Gamma_{p} \subset \Gamma_{p+1}$ for all $p$. By Theorem \ref{m3}, the action of each subgroup $\Gamma_{p}$ has an invariant probability measure $\mu_{p}$. Now, an invariant probability measure for the action of $\Gamma$ is obtained by taking an accumulation point of the sequence $(\mu_{p})_{p \in \mathbb{N}}$. 

We will deduce Theorem \ref{m3} as an immediate corollary of the following statement about Thompson's group $V_n$, which could be of independent interest:\\

\begin{thm}\label{m4} For any finitely generated subgroup  $ \Gamma \subset V_n$, either $\Gamma$ has a finite orbit or $\Gamma$ contains a free subgroup.
\end{thm}

The proof of this result involves the study of the dynamics of elements of $V_n$ on the Cantor set $K_n$ where it acts naturally. These dynamics are known to be of contracting-repelling type and the spirit of our proof is similar to the proof of Margulis for $\text{Homeo}(\s^1)$. The proof of this theorem uses the following lemma, which might be useful to prove similar results.\\

\begin{lem}\label{m5} Let $\Gamma$ be a countable group acting on a compact space $K$ by homeomorphisms and a finite subset $F \subset K$. Then there is finite orbit of $\Gamma$ on $K$ or there is an element $g\in \Gamma$ sending $F$ disjoint from itself (i.e. $g(F ) \cap F =  \emptyset $).

\end{lem}

The proof of the previous lemma is based on Horbez's recent proof of the Tits alternative for mapping class groups and related automorphisms groups (see \cite{CH1} and \cite{CH2}).

\end{subsection}

\begin{subsection}{Outline of the article}

In Section 2, we prove Theorem \ref{m1}. In Section 3, we show that the group $\mathfrak{diff}_{S}(K)$ is independent of the surface $S$ where $K$ is embedded. Then, we will focus on the study of the standard ternary Cantor set, and, in  Section 4, we prove Theorem \ref{m2}. Finally, in Section 5, we prove Theorems \ref{m3}, \ref{m4} and Lemma \ref{m5}. Section 5 is independent of Section 2 and Section 3 is independent of the other sections.

\end{subsection}

\begin{subsection}{Acknowledgements}
We would like to thank Michael Handel for pointing out Alibegovic's results could be useful for the proof of Theorem \ref{m1} (and they indeed were).
\end{subsection}

\end{section}

\begin{section}{Distorted elements in smooth mapping class groups}

In this section, we prove Theorem \ref{m1}, which we restate now:

\begin{thm}\label{main2}  Let $S$ be a closed surface and $K$ be a closed subset of $S$ which is a Cantor set. The elements of $\pmcg(S, K)$ are undistorted in $\mcg(S,K)$.
\end{thm}

The proof and main ideas of Theorem \ref{main2} come from the work of Franks and Handel about distorted elements in surface diffeomorphism groups (see \cite{FH}) and the work of Alibegovic about distorted elements in the outer automorphism group of a free group (see \cite{Ali}).\\

The reason why the hypothesis $f|_{K}= \text{Id}$ makes things more simple is the following observation.\\

\begin{lem}[Handel]\label{handel} Let $S$ be a surface. Suppose $f$ is a diffeomorphism fixing a compact set $K$ pointwise. Suppose that $K$ contains an accumulation point $p$. Then, there exists a neighborhood $U$ of $p$ and an isotopy $f_t$ fixing $K$ pointwise such that $f_0 = f$ and $f_1 |_U = \text{Id}$.

\end{lem}

\begin{proof}

We consider the homotopy $h_t = tf + (1-t)\text{Id}$ in a coordinate chart around $p$. Take a sequence of points $p_n \in K$ converging to $p$. As $f(p_n) = p_n$, there exists $v \in T_p(S)$ such that $D_pf(v)=v$. Observe that the equation $D_ph_t(w) = 0 $ implies that $D_pf(w) = -\alpha w$, for some $\alpha>0$. As $D_pf(v) = v$ and $f$ is orientation preserving, this is not possible unless $w = 0$. Hence $D_ph_t$ is invertible.\\ 

This implies that, for every $t$,  $Dh_t$ is invertible in a neighborhood $U$ of $p$: there is a neighborhood $U$ of $p$ where $h_t$ is invertible and is therefore an actual isotopy between the inclusion $i: U \to S$ and $f|_U$.\\

Now, using the Isotopy Extension Theorem (Lemma 5.8 in Milnor h-cobordism book) we can extend the isotopy $h_t|_{U}$ to an actual isotopy $g_t$ of $S$ such that $g_0 = \Id$, $g_{t}$ fixes $K$ pointwise and that coincides with $h_t$ on $U$. Therefore the isotopy $f_t = g_t^{-1}f$ gives us the desired result.
\end{proof}

 If the closed set $K$ is perfect, we can take an appropriate finite cover of $K$ by coordinate charts and use the previous lemma to prove that every element $f \in \pmcg(S, K)$  is isotopic to a diffeomorphism $f'$ which is the identity in a small neighborhood of $K$. Hence it can be considered as an element of a mapping class group of a surface of finite type.\\

\begin{cor} \label{local}

Let $S$ be a surface. Suppose $f$ is a diffeomorphism fixing a compact perfect set $K$ pointwise. Then there exists a diffeomorphism $g$ isotopic to $f$ relative to $K$ and a finite collection of smooth closed disjoint disks $\{D_i\}$ covering $K$ such that $g|_{D_i} = \Id$.
\end{cor}

As we will show next, this corollary implies that if $f$ is distorted in $\mcg(S,K)$, then $f$ must be isotopic to a composition of Dehn twists about disjoint closed curves $\beta_i$ which do not meet $K$.\\

\subsection{Representatives of mapping classes} \label{reduction}

We use the following theorem due to Franks and Handel (see Theorem 1.2, Definition 6.1 and Lemma 6.3 in \cite{FH2}). This theorem is a direct consequence of Lemma \ref{handel} and of classical Nielsen-Thurston theory (see \cite{FLP} on Nielsen-Thurston theory). Given a diffeomorphism $f$ of the surface $S$, we denote by $\mathrm{Fix}(f)$ the set of points of $S$ which are fixed under $f$.

\begin{thm}[Franks-Handel] \label{FH}
Let $f$ be a diffeomorphism in $\mathrm{Diff}_{0}^{\infty}(S)$ and $M= S- \mathrm{Fix}(f)$. There exists a finite set $R$ of disjoint simple closed curves of $M$ which are pairwise non-isotopic and a diffeomorphism $\varphi$ of $S$ which is isotopic to $f$ relative to $\mathrm{Fix}(f)$ such that:
\begin{enumerate}
\item For any curve $\gamma$ in $R$, the homeomorphism $\varphi$ preserves an open annulus neighbourhood $A_{\gamma}$ of the curve $\gamma$.
\item For any connected component $S_{i}$ of $S-\cup_{\gamma}A_{\gamma}$, 
\begin{enumerate}
\item if $\mathrm{Fix}(f) \cap S_{i}$ is infinite, then $\varphi_{|S_{i}}=Id_{|S_{i}}$. 
\item if $\mathrm{Fix}(f) \cap S_{i}$ is finite, then either $\varphi_{|M_{i}}=Id_{|M_{i}}$ or $\varphi_{|M_{i}}$ is pseudo-Anosov, where $M_{i}=S_{i}-\mathrm{Fix}(f)$.
\end{enumerate}
\end{enumerate}
\end{thm}

We need more precisely the following corollary of this theorem.

\begin{cor}\label{FH2}
Let $\xi$ be an element in $\mcgi(S,K)$ which fixes a closed set $K$ pointwise. There exists a finite set $R$ of disjoint simple closed curves of $M$ which are pairwise non-isotopic and a diffeomorphism $\psi$ of $S$ which is a representative of $\xi$ such that: 
\begin{enumerate}
\item For any curve $\gamma$ in $R$, the diffeomorphism $\psi$ preserves an open annulus neighbourhood $A_{\gamma}$ of the curve $\gamma$.
\item For any connected component $S_{i}$ of $S-\cup_{\gamma}A_{\gamma}$, 
\begin{enumerate}
\item if $K \cap S_{i}$ is infinite, then $\psi_{|S_{i}}=Id_{|S_{i}}$. 
\item if $K \cap S_{i}$ is finite, then either $\psi_{|M_{i}}=Id_{|M_{i}}$ or $\psi_{|M_{i}}$ is pseudo-Anosov, where $M_{i}=S_{i}-K$.
\end{enumerate}
\end{enumerate}
\end{cor}

\begin{proof}
Apply Theorem \ref{FH} to a representative of $\xi$. This provides a diffeomorphism $\varphi$ with the properties given by the theorem. When $K \cap S_{i}$ is infinite, the theorem states that $\varphi_{|S_{i}}=Id_{|S_{i}}$ and we take $\psi_{|S_{i}}=Id_{|S_{i}}$. If the set $K \cap S_{i}$ is finite, we can apply the classical Nielsen-Thurston theory (see Theorem 5 p.12 in \cite{FLP}) to $\varphi_{|S_{i}}$ to obtain a decomposition of $S_{i}$ and a diffeomorphism $\psi_{|S_{i}}$ whose restriction to each component of this decomposition satisfies (b).
\end{proof}

Notice that, in the case where $K$ is a Cantor set, Case (b) in Corollary \ref{FH2} can happen only if $K$ does not meet $S_{i}$. 

Let $\xi$ be an element in $\mcgi(S,K)$ and denote by $\psi$ its representative in the group $\mathrm{Diff}_{0}^{\infty}(S)$ given by the above corollary. Here, we distinguish three cases to prove Theorem \ref{main2}.
\begin{enumerate}
\item There exists a connected component $S_{i}$ for which Case (b) occurs and $\psi_{|M_{i}}$ is pseudo-Anosov.
\item The diffeomorphism $\psi$ is a composition of Dehn twists about curves of $R$ and one of the curves of $R$ is not nullhomotopic in $S$.
\item The diffeomorphism $\psi$ is a composition of Dehn twists about curves of $R$ which are nullhomotopic in $S$.
\end{enumerate}
The first case is adressed in Subsection \ref{pseudoanosov}, the second one in Subsection \ref{alibegovic} and the third one in Subsection \ref{Dehntwist}. 

\subsection{The pseudo-Anosov case} \label{pseudoanosov}

In this subsection, we need the following obstruction to being a distorted element. 

Endow the surface $S$ with a Riemannian metric $g$. For a loop $\alpha$ of $S$, we denote by $l_{g}(\alpha)$ its length with respect to the chosen Riemannian metric and, for an isotopy class $[\alpha]$ of loops of $S-K$, we denote by $l_{g}([\alpha])$ the infimum of the lengths of loops representing $[\alpha]$.

\begin{lem}\label{lengthcurves}
Let $\xi$ be an element of $\mcg(S,K)$ and $\alpha$ be a smooth loop of $S-K$. Suppose that $\xi$ is distorted in $\mcg(S,K)$. Then
$$ \lim_{n \rightarrow + \infty} \frac{\mathrm{log}(l_{g}(\xi^{n}([\alpha])))}{n}=0.$$
\end{lem}

\begin{proof}
By definition of a distorted element, there exists a finite set $\mathcal{G} \subset \mcg(S,K)$ and a sequence $(l_{n})_{n}$ of integers such that
\begin{enumerate}
\item for any $n$, $\xi^{n}=\eta_{1,n}\eta_{2,n}\eta_{3,n} \ldots \eta_{l_{n},n}$, where, for any indices $i$ and $j$, $\eta_{i,j} \in \mathcal{G}$.
\item $\lim_{n \rightarrow + \infty} \frac{l_{n}}{n}=0$.\\
\end{enumerate}

For any element $\eta$ in $\mathcal{G}$, choose a representative $g_{\eta}$ of this element in $\mathrm{Diff}^{\infty}(S)$. Also take a representative $\alpha$ of the isotopy class $[\alpha]$. Denote $M= \max_{\eta \in \mathcal{G}, x\in S} \left\|Dg_{\eta}(x) \right\|$ where $\left\|. \right\|$ is the norm associated to the chosen Riemannian metric on $S$. Then, for any $n$,
$$\begin{array}{rcl}
l(\xi^{n}([\alpha])) & \leq & l(g_{\eta_{1,n}}g_{\eta_{2,n}}g_{\eta_{3,n}} \ldots g_{\eta_{l_{n},n}}(\alpha)) \\
 & \leq & M^{l_{n}}l(\alpha).
\end{array}
$$
Hence $\frac{\mathrm{log}(l(\xi^{n}([\alpha])))}{n}\leq \frac{l_{n} \mathrm{log}(M)+\mathrm{log}(l(\alpha))}{n}$. This inequality yields the conclusion of the lemma.
\end{proof}

\begin{proof}[Proof of Theorem \ref{main2} : case of a class with a pseudo-Anosov component]
Suppose that there exists a connected component $S_{i}$ for which Case (b) in Corollary \ref{FH2} occurs. In this case, it is a classical fact that there exists an isotopy class of essential loop $[\alpha]$ such that $\lim_{n}\frac{\mathrm{log}(l(\xi^{n}([\alpha])))}{n}>0$ (see Proposition 19 p.178 in \cite{FLP}). Hence, by Lemma \ref{lengthcurves}, the element $\xi$ is undistorted in $\mcg(S,K)$.
\end{proof}

\subsection{Case of a Dehn twist about an essential curve} \label{alibegovic}

In this subsection, we show that any element in $\pmcg(S,K)$ whose normal form contains a Dehn twist about a curve which is essential in $S$ is undistorted in $\mcg(S,K)$. To prove this, we will show something slightly stronger in the $C^{0}$-topology using the method developed by Alibegovic in \cite{Ali} to show there are no distorted elements in $Out(F_n)$.\\

For a Cantor set $K$ in a surface $S$, let $\homeo(S,K)$ be the group of homeomorphisms of $S$ that preserve $K$ and $\homeo_0(S,K)$ be the identity component of $\homeo(S,K)$. We define the topological mapping class group:
$$\mcgt(S,K) = \homeo(S,K)/ \homeo_0(S,K).$$ As every Cantor set $K \subset S$ is homeomorphic to the ternary Cantor set $K'$ via a homeomorphism of $S$, we assume throughout this section that $K$ is the ternary Cantor set in $S$.\\

Take a disk $D\subset S$ containing $K$ in its interior and think of $D$ as a disk in $\R^2$ containing the ternary Cantor set $K$. Fix a point $p_{0}$ on the boundary of this disk. For any $n \geq 1$ and any sequence $I \in \{0,1\}^n$, we define the disk $D_I$ as an open disk containing exactly the points in the ternary cantor set  $K \subset \R^2$ whose tryadic expansion starts by $2I$. We choose those disks so that their boundaries are pairwise disjoint. In particular, if $n$ is fixed and $I$ and $J$ belong to $\{0,1\}^n$, the disks $D_{I}$ and $D_{J}$ are disjoint. Also, choose those disks in such a manner that the diameter of $D_{I}$ tends to $0$ when the length of $I$ tends to $+\infty$. \\

We define $S_n$ as the surface with boundary $$S_n :=  S \setminus \bigcup_{I \in  \{0,1\}^n} D_{I}$$
and 
$$U_{n}= \bigcup_{I \in \left\{ 0, 1 \right\}^{n}} D_{I}.$$
See Figure  \ref{s2}.

\begin{figure}[ht]
\begin{center}
\includegraphics[scale=0.4]{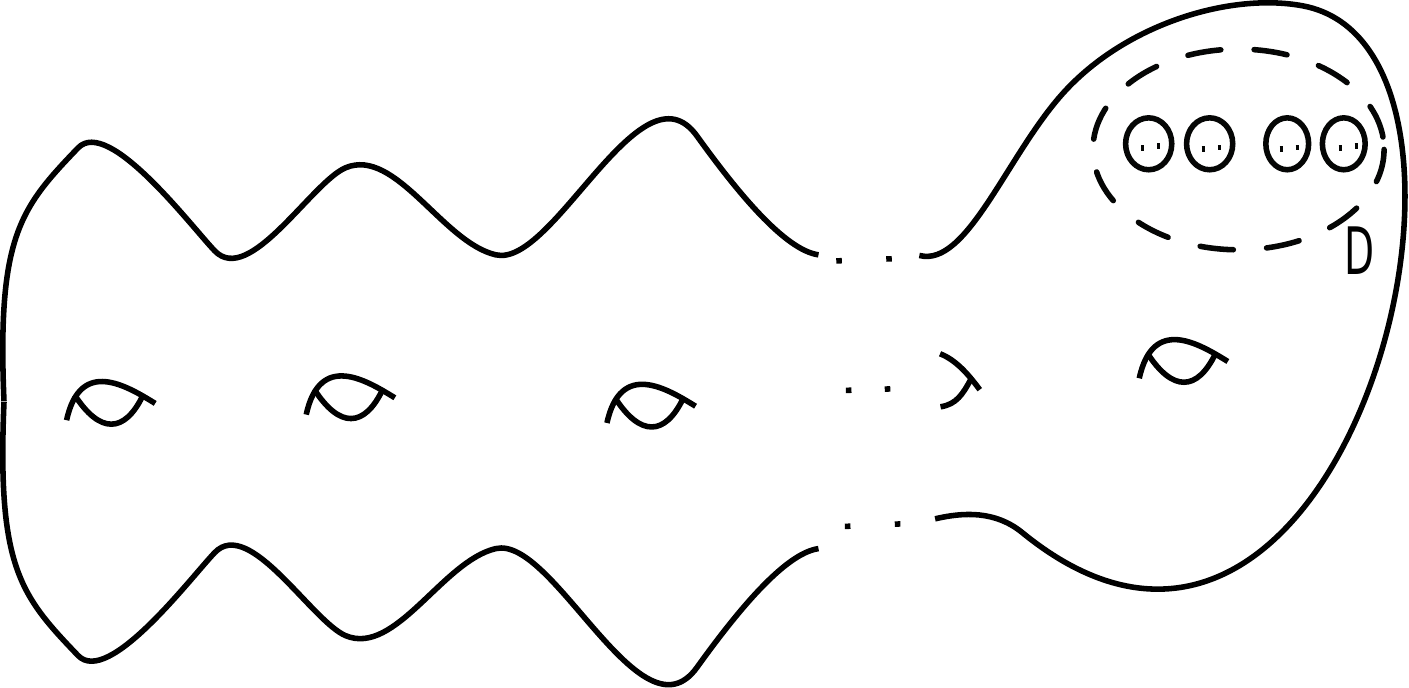}
\end{center}
\caption{The surface $S_2$}
\label{s2}
\end{figure}

The surfaces $S_n$ exhaust $S$ and have a partial order given by the containment relation, that is, for this order, $S_n < S_m$ if $S_n \subset S_m$. Observe that the group $G := \pi_1(S\setminus K, p_0)$ is the direct limit of the family of groups $(G_n, f_{n,m})$, where $G_n := \pi_1(S_n, p_0)$ and the group morphisms $ f_{n,m}: G_n \to G_m$ are obtained from the inclusions $S_n\to S_m$. Observe also that each group $G_{n}$ is a non-abelian free group.\\

\begin{figure}[ht]
\begin{center}
\includegraphics[scale=0.5]{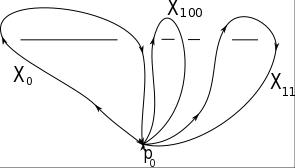}
\end{center}
\caption{The loops $x_0$, $x_{100}$ and $x_{11}$.}
\label{loops}
\end{figure}

For any sequence $I \in \{0,1\}^n$,  we consider the loop $x_I$ based at $p_0$ as in Figure \ref{loops} (it bounds a disk whose intersection with $K$ consists of elements of $K$ whose tryadic expansion starts with $2I$). We also take a standard generating set $A$ of $\pi_1(S \setminus D, p_0)$. In particular, the group $G_{0}=\pi_{1}(S-D,p_{0})$ is the free group on elements of $A$. See Figure \ref{loops2}.

\begin{figure}[ht]
\begin{center}
\includegraphics[scale=0.3]{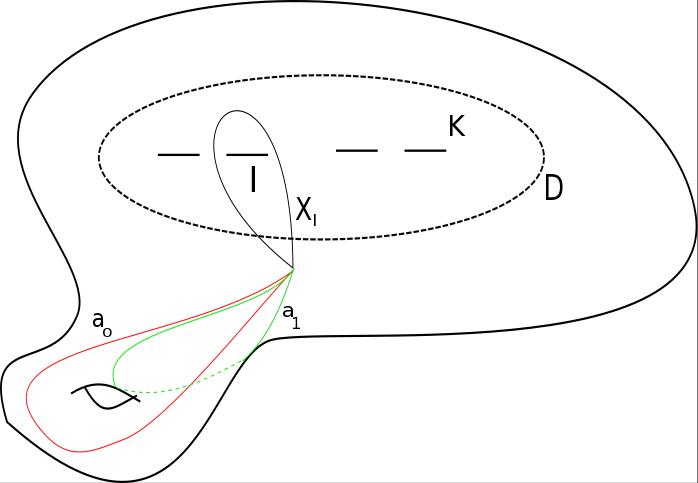}
\end{center}
\caption{$S = \mathbb{T}^2$ and $A = \{a_0.a_1\}$.}
\label{loops2}
\end{figure}


Observe that for our elements $x_I$ in $\pi_1(S \setminus K, p_0)$, the following equality holds $$x_{I} = x_{I0}x_{I1}$$

For each surface $S_n$, we consider the set of generators $X_n \cup  A$ of the fundamental group $G_n = \pi_1(S_n, p_0)$, where $$X_n = \{ x_{I} \text{ such that } I \in \bigcup_{k \leq n}{\{0,1\}}^{k} \text{ and } \text{ the last element of} \ I \ \text{is} \ 0 \}$$ Any element in $G_n$ can be written uniquely in a reduced way using these generators of $G_n$. As usual, we will identify elements of $G_{n}$ with reduced words on the elements of $A \cup X_{n}$. For any $g$ in $G_{n}$, we denote by $l_{n}(g)$ the word length of $g$ with respect to this generating set. \\

If $n < m$, recall that we have a natural inclusion $f_{n,m}: G_{n} \rightarrow G_{m}$. Observe that, for any $g \in G_{n}$, $l_{n}(g)=l_{m}(f_{n,m}(g))$ by our choice of generating set. Hence we have a well defined length $l$ on $G$, which is the direct limit of the $G_{n}$'s. Here is another way to see $l$. Observe that $G$ is the free group on the infinite set $A \cup X$, where
$$ X = \bigcup_{n \in \mathbb{N}} X_{n}.$$
The map $l: G \rightarrow \mathbb{R}_{+}$ is just the word length on $G$ with respect to this generating set.

We define the following:

\begin{defn} For each $w \in G=\pi_{1}(S\setminus K, p_0)$, we can write $w$ in a unique way using the generators $X \cup A$ of $G$. We define $c(w)$ for a reduced word $w$ as the maximal $k$ such that $$w = w_0\bar{w}^kw_1$$  where $\bar{w}$ is a word that represents a non-trivial element in $\pi_1(S)$. If no such decomposition exists, we define $c(w) = 0$.
\end{defn}

\begin{defn} For an isotopy class $[w]$ of loop in $S \setminus K$, we define $c([w])$ as the maximal $c(w')$ such that $w'$ is a cyclically reduced word representing $[w]$.
\end{defn}

For example, consider the $2$-torus $\mathbb{T}^2$ in Figure \ref{loops2}. The chosen generators of $\pi_{1}(S_n,p_{0})$ are $A = \{a_0,a_1\}$ together with the set of loops $X_n$. In that case, if $$w = (a_0x_0)^2x_0^5({a_0}a_1x_0{a_0}^{-1}{a_1}^{-1})^7$$ then $c(w) = 2$, because of the $(a_0x_0)^2$ factor. If $S$ is the $2$-sphere $\s^2$, then, for any word $w$ we have  $c(w) = 0$. Hence we will assume from now on that $g(S) \geq 1$.\\

We are in position to prove the main two propositions of this subsection.

\begin{prop}\label{disk} For any  $g \in \text{Homeo}(S,K)$ supported in $D$ and any isotopy class $[w]$ of loop in $S \setminus K$:
$$c(g([w])) \leq c([w]) + 2.$$
\end{prop}

\begin{proof}
Let $w$ be a cyclically reduced word representing a loop $[w]$ in $S \setminus K$ and let $S_n$ be a surface containing the loop $w$. We will begin the proof by making some observations on how such a $g$ supported in $D$ can act on words in the generators $X_n \cup A$ of $G$. Let $c$ be the word in $A$ corresponding to the loop surrounding the disk $D$ (a product of commutators of elements of $A$). Observe that the group $\pi_1(D,K)$ is generated by $X_n \cup \{c\}$ and in fact any element  in $\pi_1(D,K)$ can be written in a unique way as a word of the form $w(x,c)$. Observe that, if  $g \in \text{Homeo}(D,K)$, then, for any word $a \in \langle A \rangle$, we have $g(a) = a$. For a word $w(x,c)$ in $X_n \cup c$, the element $g(w(x,c))$ is a word $w'(x,c)$ in $X_l \cup c$ for some $l$ possibly larger than $n$.\\

If $w = \prod x_ja_j$, where the $x_j$'s are words in $X_n$ and the $a_j$ are words in $A$, we can rewrite $w$ if some of the $a_j's$ are powers of $c$ or if the $a_j$'s begin or end by powers of $c$. We write $w$ as a word $$w = \prod^{m}_{i=1} w_i(x,c)a_i$$ where the $a_i$'s are words in $\langle A \rangle$ such that the beginning and ending of the $a_i$'s are not powers of $c$. Applying $g$, we obtain $$g(w) = \prod_{i=1}^m g(w_i(x,c))a_i.$$ 

If $w_i'(x,c)$ is the reduced word in $X_l \cup c$ which represents $g(w_i(x,c))$, we can write $g(w) = \prod_{i=1}^m w'_i(x,c)a_i$. The word $\prod_{i=1}^m w'_i(x,c)a_i$ might not be reduced as $w'_i(x,c)$ might end or begin with powers of $c$ and some cancellations might occur at the beginning and ending of these $a_i$'s. But, after performing these cancellations, we will not be able to cancel completely any of the words $a_i$, as by assumption the $a_i$'s are not powers of $c$. In this way, we obtain a reduced word for $g(w)$ of the form $$g(w) = \prod_{i=1}^m w''_i(x,c)a'_i$$ where the words $a'_i$'s neither begin nor end with powers of $c$ and the words $w''_i(x,c)$ and $a_i'$ are non-trivial (except possibly for $a''_m$ and $w''_1(x,c)$).\\

We can now begin the actual proof of our proposition. We take a cyclically reduced word $w = C\bar{w}^mD$, such that $c([w]) = c(w) = m$ and $\bar{w}$ represents a non-trivial element in $\pi_1(S)$. If we write as before $w =\prod_{i=1}^m w_i(x,c)a_i$, where the $a_i's$ are words in $\langle A \rangle$ which neither begin nor end with powers of $c$, then, looking at the first  time $\bar{w}$ appears as a subword of the word  $\prod_{i=1}^k w(x_i,c_i)a_i$, there should be at least one $a_i$  that appears in $\bar{w}$, because $\bar{w}$ is non-trivial in $\pi_1(S)$. Let us denote this $a_i$ by  $\alpha$. We call by  $\alpha_j$, the corresponding  $\alpha$ in the $j$-th time $\hat{w}$ appears in $w$ from left to right.\\

Therefore, we can write our word $w=w_1 \big( \prod_{j=1}^{m-1}\alpha w_0 \big) \alpha w_2$, where $w_0$ is a word that begins and ends in the alphabet $X_n \cup C$. Also $w_1$ ($w_2$ respectively) is a word that ends (begins respectively) in a word of the alphabet $X_n \cup C$. Applying $g$, we get: $$g(w) = g(w_1) \bigg( \prod_{j=1}^{m-1}\alpha g(w_0) \bigg) \alpha g(w_2)$$

Let $w'_i$ be the reduced word which represents $g(w_i)$ in $X_l \cup \{c\}$. Observe that $w'_0$ begins and ends with words in the alphabet $X_l \cup C$ (Same thing for the end of $w'_1$ and the beginning of $w'_2$). At this point, some cancellations might occur at the beginning or the end of the $\alpha$'s but, for the $\alpha$'s in the word $g(w)$ that have the subword $g(w_0)$ to the left and to the right, the same cancellations should occur and so after applying these cancellations we obtain $\alpha'$, $w''_i$'s and a reduced word for $g(w)$ of the form: $$g(w) = w''_1 \bigg(\prod_{j=2}^{m-1} \alpha' w''_0 \bigg) w''_2.$$

Moreover, $\alpha'w_{0}''$, which is conjugate to $\alpha g(w_{0})$, is a nontrivial word in $A$ representing a non-trivial class in $\pi_1(S)$. We can see from there that $$c([g(w)]) \geq c([w]) - 2.$$

Applying this to $g^{-1}$ and the word $g(w)$ we obtain $$c(g(w)) \leq c(w) + 2.$$

\end{proof}

Let us define $\pmcgt(S,K)$ as the set of mapping classes in $\mcgt(S,K)$ represented by  a homeomorphism  which pointwise fixes a small neighborhood $N_\epsilon(K)$ of $K$. For any mapping class $g$ in $\pmcgt(S,K)$, we denote by $N(g)$ the smallest integer $n$ such that $g$ has a representative which pointwise fixes $U_{n}$. Finally, for any mapping class $g$ we fix a homeomorphism $\hat{g}$ which represents $g$, which fixes the point $p_{0}$ and which pointwise fixes $U_{N(g)}$. 

\begin{prop}\label{fix} For any $g \in \pmcgt(S,K)$, there exists a constant $C_g$ such that for every isotopy class $[w]$ of loop in $S\setminus K$, the following holds:
$$c(g([w])) \leq c([w]) + C_g.$$
\end{prop}

Before proving this proposition, we need two lemmas which will be proved afterwards.

For any mapping class $\varphi \in \pmcgt(S_{n})$, we denote by $\hat{\varphi}$ a homeomorphism of $S_{n}$ which fixes the point $p_{0}$. The homeomorphism $\hat{\varphi}$ induces an automorphism $\pi_{1}(S_{n},p_{0}) \rightarrow \pi_{1}(S_{n},p_{0})$ which we also denote by $\hat{\varphi}$. Let $\Lambda_{n}(\hat{\varphi})= \max_{\epsilon \in \left\{ -1,1 \right\} \ w \in A \cup X_{n}} l(\hat{\varphi}^{\epsilon}(w))$. The following lemma is essentially due to Alibegovic (see \cite{Ali}). As the constant appearing in this lemma is not explicit in \cite{Ali} and as we need an explicit constant, we will prove it.

\begin{lem} \label{Alibegovic}
For any $w \in \pi_{1}(S_{n},x_{0})$ and any $\varphi \in \pmcgt (S_{n})$,
$$c([\varphi(w)]) \geq c([w])-2 \Lambda_{n}(\hat{\varphi})^2.$$
\end{lem}

This lemma combined with the lemma below will yield Proposition \ref{fix}.

\begin{lem} \label{unifbounded}
Let $f \in \pmcgt(S,K)$. There exist $C_{f} >0$ such that, for any $w \in X \cup A$ and any $\epsilon \in \left\{ -1, 1 \right\}$,
$$ l(\hat{f}^{\epsilon}(w)) \leq C_{f}.$$ 
\end{lem}

\begin{proof}[Proof of Proposition \ref{fix}]
For any $n \geq N(g)$, the mapping class $g$ induces a mapping class in $\pmcgt (S_{n})$. By Lemma \ref{Alibegovic}, for any $w$ in $\pi_{1}(S_{n},x_{0})$,
$$c([g(w)]) \geq c(w) -2 \Lambda_{n}(\hat{g})^2.$$
By definition of $\Lambda_{n}$ and by Lemma \ref{unifbounded}, the quantity $2\Lambda_{n}(\hat{g})^2$ is uniformly bounded above (independently of $n$) by a constant $C_{g}$. Hence, for any $w$ in $\pi_{1}(S-K,p_{0})$, we have $c([g(w)]) \geq c([w])-C_{g}$. Then it suffices to apply the above inequality with the element $[g^{-1}(w)]$ instead of $[w]$ to obtain the proposition.
\end{proof}

\begin{proof}[Proof of Lemma \ref{Alibegovic}]
Let $p=c([w])$. Take a cyclically reduced representative $w$ in $\pi_{1}(S_{n},p_{0})$ of the class $[w]$ such that $c(w)=p$. We assume that $p > 2 \Lambda_{n}(\hat{\varphi})^{2}$. Otherwise, the lemma is trivial. By definition of $c$, there exists a word $\underline{w}$ in $\pi_{1}(S_{n},p_{0})$ which projects non-trivially on $\pi_{1}(S,p_{0})$ such that $w=w_{1}\underline{w}^p w_{2}$. , where $w_{1}, w_{2} \in \pi_{1}(S_{n},p_{0})$. Then
$$ \hat{\varphi}(w)= \hat{\varphi}(w_{1}) \hat{\varphi}(\underline{w})^p \hat{\varphi}(w_{2}).$$
Let $\lambda \underline{w}'\lambda^{-1}$ be the reduced representative of $\hat{\varphi}(\underline{w})$, where the word $\underline{w}'$ is cyclically reduced. Of course, the element $\underline{w}'$ projects nontrivially on $\pi_{1}(S,p_{0})$ as the morphism $\hat{\varphi}$ induces an automorphism of $\pi_{1}(S, p_{0})$. Note that the element $\hat{\varphi}(w_{2}^{-1}w_{1})\lambda \underline{w}'^{p} \lambda^{-1}$ belongs to the conjugacy class of $\hat{\varphi}(w)$ Let $w'_{12}$ be the reduced representative of $\varphi(w_{2}^{-1}w_{1})$ and let $W=w_{12} \lambda \underline{w}' \lambda^{-1}$. By Lemma \ref{Cooper} below, there are at most $\Lambda_{n}(\hat{\varphi})^2$ simplifications between the words $w'_{12}$ and $\lambda \underline{w}'^{p} \lambda^{-1}$. Moreover, by the same lemma, the length of a word $A$ such that the reduced representative of $W$ has the form $A W' A^{-1}$, with $W'$ cyclically reduced, is at most $\Lambda_{n}(\hat{\varphi})^2$. Hence the word $W'$ contains the word $\underline{w}'^{p-2 \Lambda_{n}(\hat{\varphi})^2}$ as a subword.  As a consequence $$c([\varphi(w)]) \geq c(W') \geq p-2\Lambda_{n}(\hat{\varphi})^2.$$  
\end{proof}

Let $n\geq 1$. We denote by $l$ the word length on the free group $F_{p}$ on $p$ generators and by $\mathcal{G}$ the standard generating set of this free group. For any automorphism $f$ of the free group $F_{p}$, let 
$$ \Lambda(f)=\max_{\epsilon \in \left\{ -1,1 \right\}, \ g \in \mathcal{G}} l(f^{\epsilon}(g)).$$
For elements $w$ and $w'$ of $F_{p}$, which we see as reduced words on elements of $\mathcal{G}\cup \mathcal{G}^{-1}$, we denote by $co(w,w')$ the length of the longest word $\underline{w}$ such that the words $w$ and $w'$ both begin with $\underline{w}$. The following lemma is due to Cooper (see \cite{Coo}).

\begin{lem} \label{Cooper}
Let $f$ be an automorphism of $F_{p}$. For any $w$ and $w'$ in $F_{p}$, if $co(w,w')=0$, then $co(f(w),f(w'))\leq \Lambda(f)^2$.
\end{lem}

This lemma means that, if $w^{-1}w'$ is a reduced word, then there are at most $\Lambda(f)^{2}$ simplifications in $f(w)^{-1}f(w')$ between the words $f(w)^{-1}$ and $f(w')$.

\begin{proof}
It suffices to prove that, for any elements $w$ and $w'$ of $F_{p}$,
$$co(w,w') > \Lambda(f)^{2} \Rightarrow co(f^{-1}(w),f^{-1}(w'))>0.$$

\begin{claim} Denote by $\overline{w}$ the maximal common prefix of $w$ and $w'$. Then
$co(f^{-1}(w),f^{-1}(w'))\geq l(\overline{w})/\Lambda(f)-\Lambda(f).$
\end{claim}

If we believe this claim and if $l(\overline{w})=co(w,w')>\Lambda(f)^{2}$, then $co(f^{-1}(w),f^{-1}(w'))>0$.

Now, let us prove this claim. It suffices to prove $co(f^{-1}(w),f^{-1}(\overline{w})) \geq l(\overline{w})/\Lambda(f)-\Lambda(f)$ and the same thing  for $f^{-1}(w')$. As $w$ and $w'$ play symmetric roles, it suffices to prove it for $w$.
Take any strict prefix of $w$ of the form $\overline{w}\xi$ and let $a \in A$ be the following letter in the reduced word $w$. Taking the product of $f^{-1}(\overline{w}\xi)$ with $f^{-1}(a)$ can cause at most $\Lambda(f)$ simplifications, hence
$$co(f^{-1}(\overline{w}\xi),f^{-1}(\overline{w}\xi a)) \geq  l(f^{-1}(\overline{w}\xi))- \Lambda(f).$$
By definition of $\Lambda(f)$, $l(f^{-1}(\overline{w}\xi)) \Lambda(f) \geq l(f(f^{-1}(\overline{w}\xi))) \geq l(\overline{w})$.
Hence $co(f^{-1}(\overline{w}\xi),f^{-1}(\overline{w}\xi a)) \geq l(\overline{w})/\Lambda(f)-\Lambda(f).$
But $co(f^{-1}(w),f^{-1}(\overline{w}))\geq \min_{\xi,a} co(f^{-1}(\overline{w}\xi),f^{-1}(\overline{w}\xi a))$, where the minimum is taken over all the words $\xi$ and letters $a$ such that $\overline{w}\xi a$ is a prefix of $w$. This yields the lemma.
\end{proof}

It suffices to prove Lemma \ref{unifbounded} to complete the proof of Proposition \ref{fix}.

\begin{proof}[Proof of Lemma \ref{unifbounded}]
For any finite sequence $I$  of $0$ and $1$, we denote by $Y_{I}$ the set of elements $w$ of $X$ such that there exists a finite sequence $J$ of $0$ and $1$ with $w=x_{IJ}$. Lemma \ref{unifbounded} is an easy consequence of the following claim.

\begin{claim} For any sequence $I$ in $\left\{ 0,1 \right\}^{N(f)}$, there exists an element $\gamma_{I}$ in $\pi_{1}(S_{N(f)},p_{0})$ such that
$$ \forall w \in Y_{I}, \hat{f}(w)=\gamma_{I}w\gamma_{I}^{-1}.$$
\end{claim}

Before proving the claim, let us prove Lemma \ref{unifbounded}. Observe that
$$ A \cup X= A \cup X_{N(f)} \cup \bigcup_{I \in \left\{ 0,1 \right\}^{N(f)}} Y_{I}.$$
By the above claim, for any element $w$ of $X \cup A$,
$$ l(\hat{f}(w)) \leq \max(1+2 \max_{I \in \left\{ 0,1 \right\}^{N(f)}} l(\gamma_{I}), \max_{w \in A \cup X_{N(f)}} l(\hat{f}(w))).$$
Moreover, by the claim and as, for any element $w$ of $X \cup A$,
$$\hat{f}^{-1}(w)=\hat{f}^{-1}(\gamma_{I}^{-1})\hat{f}^{-1}(\hat{f}(w))\hat{f}^{-1}(\gamma_{I})=\hat{f}^{-1}(\gamma_{I}^{-1})w\hat{f}^{-1}(\gamma_{I}).$$
Therefore, we can take
$$ C_{f}= \max(1+2 \max_{I \in \left\{ 0,1 \right\}^{N(f)} \ \epsilon \in \left\{-1, 0 \right\}} l(\hat{f}^{\epsilon}(\gamma_{I})), \max_{w \in A \cup X_{N(f)} \ \epsilon \in \left\{ -1,1 \right\}} l(\hat{f}^{\epsilon}(w))).$$

Now, let us prove the claim. For any $I \in \left\{ 0,1 \right\}^{N(f)}$, fix a simple curve $\alpha_{I}$ such that $\alpha_{I}(0)=p_{0}$, $\alpha_{I}(1) \in \partial D_{I}$ and $\alpha_{I}((0,1)) \subset S_{N(f)}$. A loop which represents $w \in Y_{I}$ can be written $w=\alpha_{I} w' \alpha_{I}^{-1}$, where $w':[0,1] \rightarrow S$ is a simple loop which is contained in $D_{I}$ with $w'(0)=w'(1)=\alpha_{I}(1)$. As the diffeomorphism $\hat{f}$ fixes the point $p_{0}$ and pointwise fixes $U_{N(f)}$,
$$\hat{f}(w)=\hat{f}(\alpha_{I})\hat{f}(w') \hat{f}(\alpha_{I}^{-1})=\hat{f}(\alpha_{I})w' \hat{f}(\alpha_{I}^{-1}),$$ $\hat{f}(\alpha_{I}(0))=p_{0}=\alpha_{I}(0)$, $\hat{f}(\alpha_{I}(1))=\alpha_{I}(1)$ and $\hat{f}(\alpha_{I}((0,1))) \subset S_{N(f)}$. Hence $\hat{f}(w)=\hat{f}(\alpha_{I}) \alpha_{I}^{-1} w \alpha_{I} f(\alpha_{I}^{-1}).$
It suffices to take $\gamma_{I}=f(\alpha_{I})\alpha_{I}^{-1}$.
\end{proof}

\begin{defn}Given an arbitrary cantor set $K$ in a surface $S$, there is a homeomorphism $h_K$ of $S$ sending $K$ to the standard ternary cantor set $K' \subset S$. We define, for each isotopy class $[w]$ of loop in $S \setminus K$: $$c([w]) = c([h_K(w)])$$
\end{defn}

We are ready to prove and state the main proposition for the smooth mapping class group $\mcg(S,K)$.

\begin{prop}\label{ali} Let $K$ be any cantor set in a surface $S$. For every $g \in \mcg(S,K)$, there exists a constant $C_g$ such that, for any isotopy class $[w]$ of loop contained in $S \setminus K$,
$$c(g([w])) \leq c([w]) + C_g.$$
\end{prop}

\begin{proof}

By making use of Proposition \ref{representative}, we can write any representative of a mapping class group $g \in \mcg(S,K)$ as a product $g = g_1g_2$, where $g_1$ is supported in a disk $D$ containing $K$ and $g_2$ is a pure mapping class in $\pmcg(S,K)$.\\

If $h_K$ is the homeomorphism of $S$ sending $K$ to the ternary Cantor set $K'$, we have $$ g^{h_K}=g_1^{h_K}g_2^{h_K}$$

where the notation $x^y$ is used to denote the conjugate $yxy^{-1}$.\\

Recall that a representative of $g_{2}$ pointwise fixes a neighborhood of $K$, by Lemma \ref{handel}. Hence the elements  $g_1^{h_K}$  and $g_2^{h_K}$ satisfies the conclusion of our theorem by Prop. \ref{disk} and Prop. \ref{fix}, therefore the mapping class $g^{h_K}$ in $\mcgt(S,K')$ also satisfies the conclusion of our theorem. So, for any isotopy class $[w]$ of loop contained in $S \setminus K$: 

$$c([g(w)]) = c([h_Kg(w)]) = c(g^{h_K}([h_K(w)])) \leq c([h_K(w)]) + C_{g} = c([w]) +  C_{g}.$$
\end{proof}

As a corollary, we can prove a second case of Theorem \ref{main2}.

\begin{proof}[Proof of Theorem \ref{main2} : second case]
Let $g$ be an element of $\pmcg(S,K)$ whose normal form contains a Dehn twist about a simple loop $\alpha$ which is not homotopically trivial in $S$. If such a loop $\alpha$ exists, one can find a non trivial loop  $[w] \in S\setminus K$, such that $c(g^n([w])) \geq nk_1$ for some $k_1 >0$. By Prop. \ref{ali}, the mapping class $g$ is not distorted.

\end{proof}

\subsection{All the curves of $R$ are nullhomotopic in $S$} \label{Dehntwist}

To prove Theorem \ref{main2} in this case, we need the concept of spread. This concept was introduced by Franks and Handel in \cite{FH}.

\subsubsection{Spread}

Let $\gamma$ be a smooth curve with endpoints $p, q$ (smooth at the endpoints) and $\beta$ be a simple nullhomotopic closed curve on $S$ such that the point $p$ is contained in the disk bounded by $\beta$. For any curve $\alpha$, the spread $L_{\beta, \gamma}(\alpha)$ is going to measure how many times $\alpha$ rotates around $\beta$ with respect to $\gamma$.\\

More formally $L_{\beta, \gamma}(\alpha)$ is defined as the maximal number $k$, such that there exist subarcs $\alpha_0 \subset \alpha $, $\gamma_0 \subset \gamma$ such that $\overline{\gamma_0\alpha_0}$ is a closed curve isotopic to $\beta^k$ in $S \setminus \{p,q\}$.\\

\begin{exmp} In the example depicted in Figure \ref{fig}, we have a thrice-punctured sphere $S$ together with the curves $\alpha$, $\beta$ and $\gamma$. In this case we have $L_{\beta, \gamma}(\alpha) = 3$. The bold loop is a loop isotopic to $\beta^{3}$.

\begin{figure}

  \centering
    \includegraphics[width=1.0\textwidth]{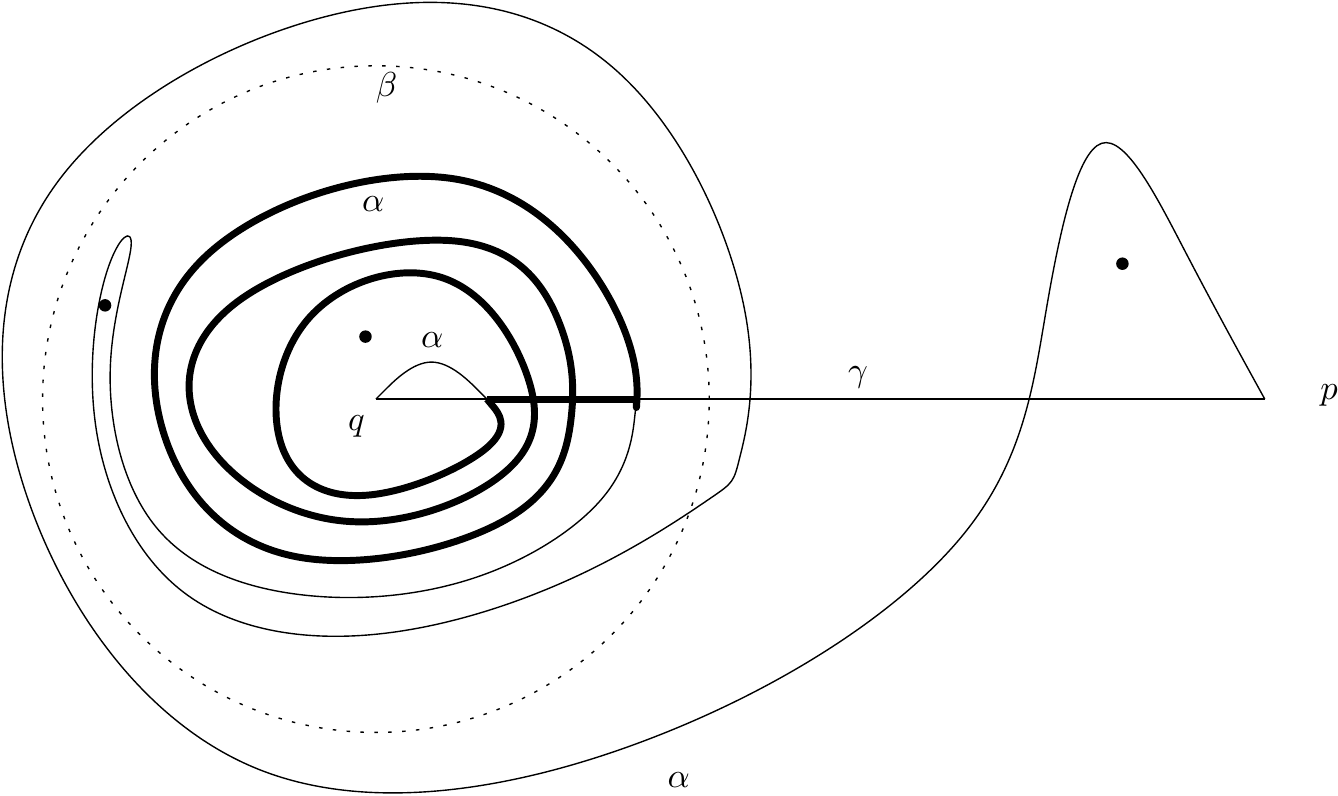}
    \caption{Examples of curves $\alpha$, $\beta$ and $\gamma$ with $L_{\beta,\gamma}(\alpha)=3$}
\label{fig}
    
\end{figure}

\end{exmp}

Let us denote by $\mathcal{C}_{S,K}$ the set of simple smooth curves $[0,1] \rightarrow S$ whose endpoints belong to $K$ and whose interior is contained in $S-K$. Two such curves are said to be isotopic if there exists a diffeomorphism in $\mathrm{Diff}_{0}(S,K)$ which sends one of these curves to the other one.

We denote by $\overline{\mathcal{C}}_{S,K}$ the set of isotopy classes of curves in $\mathcal{C}_{S,K}$. Take an isotopy class $[\alpha]$ in $\overline{\mathcal{C}}_{S,K}$. We define $\overline{L}_{\beta, \gamma}([\alpha])$ as the infimum of $L_{\beta, \gamma}(\alpha)$ over all the representatives $\alpha$ of the class $[\alpha]$.\\ 

The lemma below is stronger than Lemma 6.8 in \cite{FH} but it follows from the proof of Lemma 6.8. Notice that we state it only in the case where the curve $\beta$ bounds a disk as we believe that the proof given in \cite{FH} covers only this case (which is sufficient for the purposes of the article \cite{FH}).

Let $\gamma$ be a smooth curve in $\mathcal{C}_{S,K}$ which join the points $x_{1}$ and $x_{2}$. Let $\beta$ be a simple loop contained in $S-K$ which bounds a disk $D$ in the surface $S$ and which is homotopic to a small loop around $x_{1}$ relative to $\left\{ x_{1}, x_{2} \right\}$. We suppose that $\gamma$ is not homotopic relative to its endpoints to a curve contained in $D$.

\begin{lem}[Franks-Handel]
Let $\mathcal{G}=\left\{ g_{i}, \ 1 \leq i \leq k \right\}$ be a finite set of elements of $\mathrm{Diff}(S)$. There exists a constant $C>0$ such that the following property holds. Let $f$ be any diffeomorphism which fixes $K$ pointwise and which belongs to the group generated by the $g_{i}$'s. Then, for any curve $\alpha$ in $\mathcal{C}_{S,K}$,
$$ L_{\beta, \gamma}(f(\alpha)) \leq L_{\beta, \gamma}(\alpha)+C l_{\mathcal{G}}(f).$$ 
\end{lem}

\begin{cor}
Let $\overline{\mathcal{G}}= \left\{ \xi_{i}, \ 1 \leq i \leq k \right\}$ be a finite set of elements in $\mcg(S,K)$. Let $\eta$ be an element of the group generated by the $\xi_{i}$'s which fixes $K$ pointwise. Then there exists a constant $C>0$ such that, for any $n>0$ and $[ \alpha ]$ in $\overline{\mathcal{C}}_{S,K}$,
$\overline{L}_{\beta,\gamma}(\eta^{n}([\alpha])) \leq \overline{L}_{\beta,\gamma}([\alpha])+C l_{\overline{\mathcal{G}}}(\eta^{n})$.
\end{cor}

\begin{proof}[Proof of the corollary]
Let $l_{n}= l_{\overline{\mathcal{G}}}(\eta^{n})$. Take a representative $f$ in $\mathrm{Diff}^{\infty}(S)$ of $\eta$ and, for each $i$, choose a representative $g_{i}$ of $\xi_{i}$. For any  curve $\alpha$  representing a class $[\alpha]$ in $\mathcal{C}_{S,K}$, the curve $f^{n}(\alpha)$ represents the class $\eta^{n}([\alpha])$. Additionally, by hypothesis, we can write $f^{n}=g_{i_{1}}g_{i_{2}} \ldots g_{i_{l_{n}}}h'$, where $1 \leq i_{j} \leq k$ and $h'$ is a diffeomorphism of $S$ isotopic to the identity relative to $K$. Franks and Handel's lemma implies that there exists a constant $C>0$ independent of $n$ and $\alpha$ such that
$$L_{\beta, \gamma}(f^{n}(\alpha)) \leq L_{\beta, \gamma}(h'(\alpha))+C l_{n}.$$
Hence
$$\overline{L}_{\beta, \gamma}(\eta^{n}([\alpha])) \leq L_{\beta, \gamma}(h'(\alpha))+C l_{n}.$$
As $\alpha$ is any curve in the isotopy class of $[\alpha]$,
$$\overline{L}_{\beta, \gamma}(\eta^{n}([\alpha])) \leq \overline{L}_{\beta, \gamma}([\alpha])+C l_{n}.$$
\end{proof}

The above corollary immediately yields the result below.

\begin{cor} \label{spread}
Let $\eta$ be a distorted element in $\pmcg(S,K)$. Then, for any $[\alpha]$ in $\overline{\mathcal{C}}_{S,K}$,
$$\lim_{n \rightarrow + \infty} \frac{\overline{L}_{\beta,\gamma}(\eta^{n}([\alpha]))}{n}=0.$$
\end{cor}

We are now ready to start the proof of Theorem \ref{main2}.

\subsubsection{End of the proof of Theorem \ref{main2}}

The following proposition completes the proof of Theorem \ref{main2}.

\begin{prop} \label{dehntwwists}
Let $\xi$ be a non-trivial element of $\mcg(S,K)$. Suppose that $\xi$ fixes $K$ pointwise and is equal to a finite composition of Dehn twists about disjoint simple loops of $S-K$ which are pairwise non-isotopic relative to $K$, which are not homotopic to a point relative to $K$ and which are homotopically trivial in $S$. Then $\xi$ is not distorted in $\mcg(S,K)$.
\end{prop}
 
\begin{proof}
Denote by $f$ a representative of $\xi$ in $\mathrm{Diff}^{\infty}(S)$ which is equal to the identity outside small tubular neighbourhoods of the loops appearing in the decomposition of $\xi$. Also denote by $\mathcal{C}(f)$ the set of loops appearing in this decomposition. By hypothesis, any curve in $\mathcal{C}(f)$ bounds a disk in $S$.\\

\underline{First case:} Suppose that at least two connected components of the complement of the curves of $\mathcal{C}(f)$ contain points of $K$.

\begin{lem} \label{cases}
There exists a simple smooth curve $\alpha : [0,1] \rightarrow S$ such that its endpoints $\alpha(0)=x_{1}$ and $\alpha(1)=x_{2}$ belong to the Cantor set $K$ and with one of the following properties.
\begin{enumerate}
\item The curve $\alpha$ meets exactly one loop $\beta$ of $\mathcal{C}(f)$ transversely in only one point. Moreover, this loop $\beta$ is isotopic to a small loop around $x_{1}$ relative to $\left\{x_{1},x_{2} \right\}$.
\item The curve $\alpha$ meets exactly two loops $\beta$ and $\gamma$ of $\mathcal{C}(f)$ transversally and in one point. Moreover, the loop $\beta$ is isotopic to a small loop around $x_{1}$ relative to $\left\{ x_{1}, x_{2} \right\}$ and the loop $\gamma$ is isotopic to a small loop around $x_{2}$ relative to $\left\{ x_{1}, x_{2} \right\}$.
\end{enumerate}
\end{lem}

\begin{proof}
In order to carry out this proof, we have to introduce some notation. Suppose that the surface $S$ is different from the sphere. Then, for any loop $\gamma$ in $\mathcal{C}(f)$, there exists a unique connected component of $S-\gamma$ which is homeomorphic to an open disk. We call this connected component the interior of $\gamma$. In the case where $S$ is a sphere, we fix a point on this surface which does not belong to any curve of $\mathcal{C}(f)$. Then, for any loop $\gamma$ of $\mathcal{C}(f)$, we call interior of $\gamma$ the connected component of $S-\gamma$ which does not contain this point.\\  

To each curve $\gamma$ in $\mathcal{C}(f)$, we will associate a number $l(\gamma) \in \mathbb{N}$ which we call its level. For any curve $\gamma$ of $\mathcal{C}(f)$ which is not contained in the interior of a loop of $\mathcal{C}(f)$, we set $l(\gamma)=0$. Denote by $\mathcal{C}_{i}(f)$ the set of curves of $\mathcal{C}(f)$ whose level is $i$. The following statement defines inductively the level of any curve in $\mathcal{C}(f)$: for any $i \geq 1$, a curve $\gamma$ in $\mathcal{C}(f)- \mathcal{C}_{i-1}(f)$ satisfies $l(\gamma)=i$ if it is not contained in the interior of a curve in $\mathcal{C}(f)- \mathcal{C}_{i-1}(f)$. See Figure \ref{level} for an example. The represented loops are the elements of $\mathcal{C}(f)$ in this example and the numbers beside them are their levels.

\begin{figure}[ht]
\begin{center}
\includegraphics[scale=0.5]{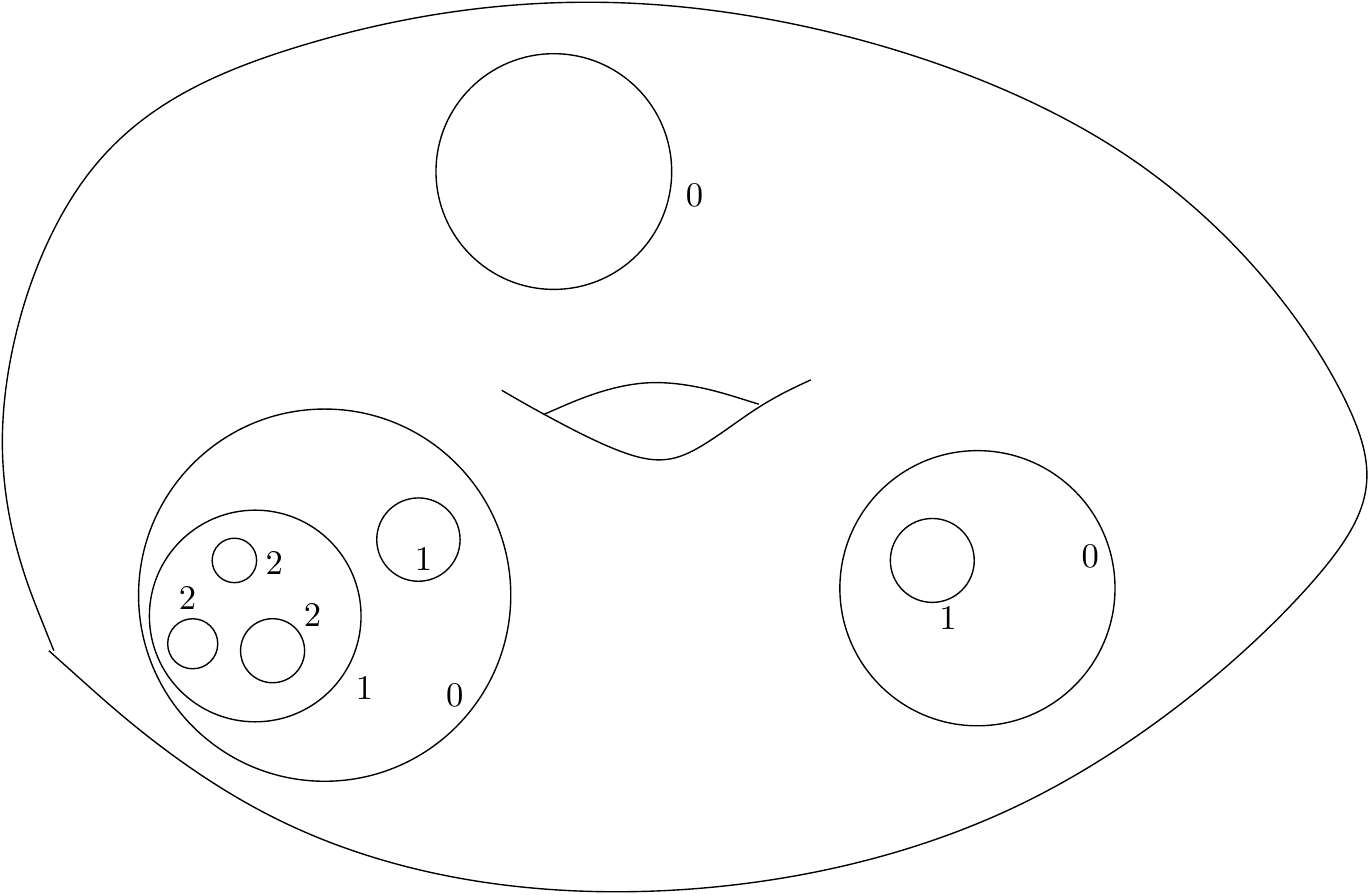}
\end{center}
\caption{Level of curves of $\mathcal{C}(f)$.}
\label{level}
\end{figure}

Now, we are ready for the proof. Take a loop $\beta$ in $\mathcal{C}(f)$ whose level is maximal. As the level of this curve is maximal, its interior does not contain any curve of $\mathcal{C}(f)$. However, its interior has to contain a point $x_{1}$ of $K$: otherwise, the loop $\beta$ would be homotopically trivial.\\ 

If $l(\beta)=0$, denote by $U$ the complement of the interiors of the curves of level $0$ in $\mathcal{C}(f)$. Otherwise, let $U$ be the complement of the interiors of the curves of level $l(\beta)$ in the interior of $\delta$, where $\delta$ is the unique loop of level $l(\beta)-1$ whose interior contains $\beta$.\\

If the open set $U$ contains some point $x_{2}$ of $K$, it is not difficult to find a curve $\alpha$ which satisfies the first property of the lemma. Otherwise, there exists a loop $\gamma$ of level $l(\beta)$  which is different from $\beta$. If $l(\beta)=0$, this is a consequence of the hypothesis made for the first case. If $l(\beta)\neq 0$, we can further require that this loop is contained in the interior of $\delta$: if this was not the case, $\beta$ would be isotopic to $\delta$, a contradiction. In this case, take a path $\alpha$ going from $x_{1}$ to $U$ crossing $\beta$ once and from $U$ to a point $x_{2}$ of $K$ contained in the interior of $\gamma$ crossing $\gamma$ once. This path satisfies the second property of the lemma.
\end{proof}

In this case, we will prove Proposition \ref{dehntwwists} only in the case where the surface $S$ is different from the sphere. The case of the sphere is similar and is left to the reader. Denote by $S'$ the surface obtained from $S$ by blowing up the points $x_{1}$ and $x_{2}$ (replacing these points with circles). Denote by $C_{1} \subset S'$ the circle which projects to the point $x_{1}$ in $S$ and by $C_{2} \subset S'$ the circle which projects to the point $x_{2}$. Denote by $\tilde{S}'$ the universal cover of $S'$. The space  $\tilde{S}'$ can be seen as a subset of the universal cover $\mathbb{H}^{2}$ of the double of $S'$, which is the closed surface obtained by gluing $S'$ with itself by identifying their boundaries. The curve $\alpha$ lifts to a smooth curve $\alpha'$ of $S'$ whose endpoints are $x'_{1} \in C_{1}$ and $x'_{2} \in C_{2}$. By abuse of notation, we denote by $\beta$ the lift of the curve $\beta$ to $S'$.\\

As the diffeomorphism $f$ fixes both points $x_{1}$ an $x_{2}$, it induces a smooth diffeomorphism $f'$ of $S'$, acting by the differential of $f$ on the circles $C_{1}$ and $C_{2}$. Observe that the diffeomorphism $f'$ preserves $C_{1}$ and $C_{2}$ and is not a priori isotopic to the identity. As the compact set $K$ is perfect, the point $x_{1}$ is accumulated by points of $K$ which are fixed under $f$. Denote by $z$ a point of $C_{1}$ which corresponds to a direction in $ T_{x_1}(S)$ accumulated by points of $K$. Then the diffeomorphism $f'$ fixes the point $z$.\\

Denote by $\tilde{C}_{1}$ a lift of $C_{1}$, \emph{i.e.} a connected component of $p^{-1}(C_{1})$, where $p: \tilde{S}' \rightarrow S'$ is the projection, and by $\tilde{z}$ a lift of the point $z$ contained in $\tilde{C}_{1}$. Denote by $\tilde{\alpha}'$ a lift of $\alpha'$ which meets $\tilde{C}_{1}$ and by $\tilde{\beta}$ the lift of $\beta$ which meets $\tilde{\alpha}'$. The curve $\tilde{\alpha}'$ joins the points $\tilde{x}_{1}' \in \tilde{C}_{1}$ and $\tilde{x}_{2}' \in \tilde{C}_{2}$, where $\tilde{C}_{2}$ is a connected component of $p^{-1}(C_{2})$. Denote by $\tilde{f}': \tilde{S}' \rightarrow \tilde{S}'$ the lift of $f'$ which fixes the point $\tilde{z}$. This lift $\tilde{f}'$ preserves $\tilde{C}_{1}$. Denote by $T$ the deck transformation corresponding to $\tilde{\beta}$ (or equivalently $\tilde{C}_{1}$). Then the diffeomorphism $\tilde{f}'$ fixes all the points of the orbit of $\tilde{z}$ under $T$. Hence the orbit of $\tilde{x}'_{1}$ under $\tilde{f}'$ lies on $\tilde{C}_{1}$ between two consecutive points of the orbit of $\tilde{z}$ under $T$. Up to translating $\tilde{z}$, we can suppose that these points are $\tilde{z}$ and $T(\tilde{z})$.\\

\begin{figure}[ht]
\begin{center}
\includegraphics[scale=0.5]{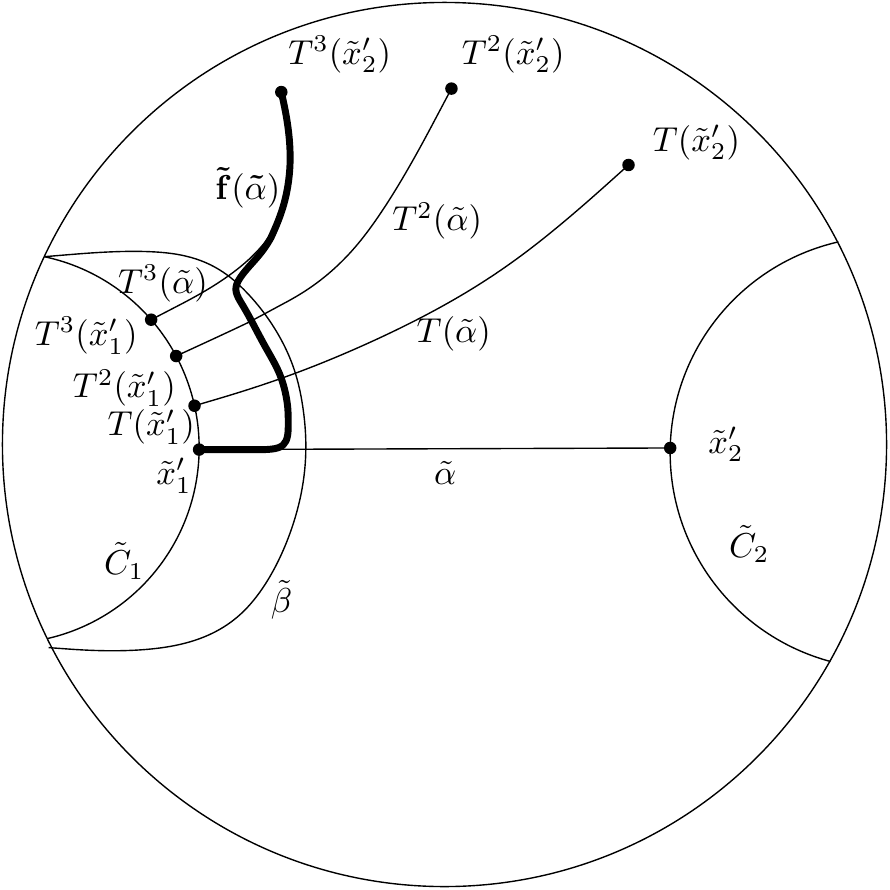}
\end{center}
\caption{$\tilde{f}(\tilde{\alpha})$ in the case $k=3$}
\label{case211}
\end{figure}

First suppose that the only loop in $\mathcal{C}(f)$ met by the curve $\alpha$ is $\beta$. In this case, by definition of a Dehn twist, there exists $k \neq 0$ such that, for any integer $n$, $\tilde{f}'^{n}(\tilde{x}'_{2}) \in T^{kn}(\tilde{C}_{2})$ (see Figure \ref{case211}). For any $1 \leq j \leq kn-1$, the curve $T^{j}(\tilde{\alpha}')$ separates $\tilde{S}'$ into two connected components. Observe that the point $\tilde{f}^{n}(\tilde{x}'_{1})$ belongs to one of them and that the component $T^{kn}(C_{2})$ belongs to the other one. Hence the curve $\tilde{f}'^{n}(\tilde{\alpha}')$ has to meet each of the curves $T^{j}(\tilde{\alpha}')$, for $1 \leq j \leq kn-1$. Denote by $\tilde{y}_{1}$ the first intersection point of the curve $\tilde{f}'^{n}(\tilde{\alpha})$ with the curve $T(\tilde{\alpha}')$ and by $\tilde{y}_{2}$ the last intersection point of $\tilde{f}'^{n}(\tilde{\alpha}')$ with the curve $T^{kn-1}(\tilde{\alpha}')$.\\

As the curve $T(\tilde{\alpha}')$ belongs to the same component of the complement of $T^{kn-1}(\tilde{\alpha}')$ as the point $\tilde{x}'_{1}$, the curve $\tilde{\alpha}'$ meets the point $\tilde{y}_{1}$ before meeting the point $\tilde{y}_{2}$. Consider the curve $\tilde{\gamma}$ which is the concatenation of the segment of $\tilde{f}'^{n}(\tilde{\alpha}')$ between $\tilde{y}_{1}$ and $\tilde{y}_{2}$ and the segment of the curve $T^{kn-1}(\tilde{\alpha}')$ between $\tilde{y}_{2}$ and $T^{kn}(\tilde{y}_{1})$ The projection on $S$ of $\tilde{\gamma}$ is isotopic to $\beta^{kn-2}$ relative to $x_{1}$ and $x_{2}$. Hence $L_{\beta,\alpha}(f^{n}(\alpha)) \geq \left|k \right| n-2$.\\ 

In the above proof, we just used the fact that the curve $\tilde{\alpha}'$ joined a point of the boundary component $\tilde{C}_{1}$ which is between $\tilde{z}$ and $T(\tilde{z})$ to the boundary component $\tilde{C}_{2}$. We will prove the following claim.\\

\begin{claim} \emph{Any curve $\delta$ of $\mathcal{C}_{S,K}$ which represents the same class as $\alpha$ in $\overline{\mathcal{C}}_{S,K}$ has a lift $\tilde{\delta}'$ to $\tilde{S}'$ with the following property. The curve $\tilde{\delta}'$ joins a point of the boundary component $\tilde{C}_{1}$ which is between $\tilde{z}$ and $T(\tilde{z})$ to the boundary component $\tilde{C}_{2}$.}
\end{claim}

Hence, for any curve $\delta$ isotopic to $\alpha$ relative to $K$, we obtain that, for any $n$, $L_{\beta, \alpha}(f^{n}(\delta)) \geq kn-2$ and $\overline{L}_{\beta,\alpha}(\xi^{n}([\alpha])) \geq \left| k \right| n-2$. By Corollary \ref{spread}, the element $\xi$ is undistorted in the group $\mcg(S,K)$.\\

\begin{proof}[Proof of the claim]
By definition, there exists a continuous path $(g_{t})_{t \in [0,1]}$ of diffeomorphisms in $\mathrm{Diff}_{0}^{\infty}(S,K)$ such that $g_{0}=Id_{S}$ and $g_{1}(\alpha)=\delta$. This path of diffeomorphisms lifts (by acting by the differential on $C_{1}$ and $C_{2}$) to a path of diffeomorphisms $(g'_{t})_{t \in [0,1]}$ of $S'$ which in turn lifts to a path $(\tilde{g}'_{t})_{t \in [0,1]}$ of diffeomorphisms such that $\tilde{g}'_{0}=Id_{\tilde{S}'}$.\\

As the direction in $S$ corresponding to $z \in C_1$ is accumulated by points of $K$, we have that for any $t$, $g'_{t}(z)=z$, then, for any $t$, $\tilde{g}'_{t}(\tilde{z})=\tilde{z}$ and $\tilde{g}'_{t}(T(\tilde{z}))=T(\tilde{z})$. Moreover, for any $t$, the diffeomorphism $\tilde{g}'_{t}$ preserves $\tilde{C}_{1}$ and $\tilde{C}_{2}$. Therefore, for any $t$, the point $\tilde{g}'_{t}(\tilde{\alpha}'(0))$ lies between the points $\tilde{z}$ and $T(\tilde{z})$ and the point $\tilde{g}'_{t}(\tilde{\alpha}'(1))$ belongs to $\tilde{C}_{2}$. In particular, the curve $\tilde{g}'_{1}(\tilde{\alpha}')$ is a lift of $\delta$ which satisfies the required properties.
\end{proof}

Now, suppose that the curve $\alpha$ meets a loop of $\mathcal{C}(f)$ which is different from $\beta$, \emph{i.e.} the second case in Lemma \ref{cases} occurs. Denote by $\tilde{\gamma}$ a lift of $\gamma$ to $\tilde{S}'$. Observe that there exists $k \neq 0$ such that, for any $n$, the curve $\tilde{f}'^{n}(\tilde{\alpha}')$ meets $T^{kn}(\tilde{\gamma})$. As the curves $T^{j}(\tilde{\alpha})$ separate the point $\tilde{x}_{1}$ from the curve $T^{kn}(\tilde{\gamma})$ for any $1 \leq j \leq kn-1$, this case can be handled in the same way as the case where the curve $\alpha$ meets only one loop in $\mathcal{C}(f)$.\\

\underline{Second case:} Only one component of the complement of the curves of $\mathcal{C}(f)$ meets $K$, \emph{i.e.} the diffeomorphism $f$ is some power of a Dehn twist about this curve. In this case, we can take this curve as $\beta$ and we take any simple curve $\alpha:[0,1] \rightarrow K$ in $\mathcal{C}_{S,K}$ with the following properties.
\begin{enumerate}
\item $\alpha(0)=\alpha(1)$.
\item The curve $\alpha$ is not homotopic relative to its endpoints to a curve contained in the disk bounded by $\beta$.
\item The curve $\alpha$ meets $\beta$ in exactly two points. 
\end{enumerate}
This case is then handled in the same way as the first case.
\end{proof}
\end{section}

\begin{section}{Independence of the surface}

Let $S$ be a closed surface and $K$ be any closed subset of $S$. Recall that $\mathfrak{diff}_{S}(K)=\mcg(S, K)/\pmcg(S, K)$. This group can also be seen as the quotient of the group $\mathrm{Diff}^{\infty}(S, K)$ consisting of diffeomorphisms that preserve $K$ by the subgroup of $\mathrm{Diff}^{\infty}(S, K)$  consisting of elements which fix $K$ point-wise.\\

In what follows, we call disk the image of the unit closed disk $\mathbb{D}^{2}$ under an embedding $\mathbb{D}^{2} \hookrightarrow S$.\\

Suppose that the closed set $K$ is contained in the interior of a disk $D$. The goal of this section is to prove the following proposition.\\

\begin{prop} \label{representative}
Any element of $\mathfrak{diff}_{S}(K)$ has a representative in $\mathrm{Diff}_{0}^{\infty}(S)$ which is supported in $D$.
\end{prop}

The proof of this proposition also implies that $\mathfrak{diff}_{S}(K)$ is isomorphic to  $\mcg(S,K)/\pmcg(S,K)$.\\

The following corollary implies that, to prove something about groups of the form $\mathfrak{diff}_{S}(K)$, it suffices to prove it in the case where the surface $S$ is the sphere.

\begin{cor} \label{invariance}
Let $S$ an $S'$ be surfaces and $K \subset S$ and $K' \subset S'$ be closed subsets. Suppose that there exist disks $D_{S} \subset S$ and $D_{S'} \subset S'$ as well as a diffeomorphism $\varphi : D_{S} \rightarrow D_{S'}$ such that:
\begin{enumerate}
\item The closed set $K$ is contained in the interior of the disk $D_{S}$.
\item The closed set $K'$ is contained in the interior of the disk $D_{S'}$.
\item $\varphi(K)=K'$.
\end{enumerate}
Then, the group $\mathfrak{diff}_{S}(K)$ is isomorphic to $\mathfrak{diff}_{S'}(K')$.
\end{cor}

\begin{proof}
Let $\psi: \mathbb{D}^{2} \rightarrow D_{S}$ be a diffeomorphism. This diffeomorphism induces a morphism
$$\begin{array}{rrcl}
\Psi: & \mathfrak{diff}_{\mathbb{D}^{2}}(\psi^{-1}(K)) & \rightarrow & \mathfrak{diff}_{S}(K) \\
 & \xi & \mapsto & \psi \xi \psi^{-1} \\
\end{array}
.$$
By Proposition \ref{representative}, this morphism is onto. Let us prove that it is into. Let $\xi$ and $\xi'$ be elements of $\mathfrak{diff}_{\mathbb{D}^{2}}(\psi^{-1}(K))$ such that $\Psi(\xi)=\Psi(\xi')$. Take representatives $f$ and $f'$ of $\xi$ and $\xi'$ in $\mathrm{Diff}_{0}^{\infty}(\mathbb{D}^{2})$. Then there exists a diffeomorphism $g$ in $\mathrm{Diff}_{0}^{\infty}(S)$ which fixes $K$ pointwise such that $\psi f \psi^{-1}=g \psi f' \psi^{-1}$. The support of $g$, which is also the support of $\psi f f'^{-1} \psi^{-1}$, is contained in $D_{S}$. Hence, there exists a diffeomorphism $g'$ in $\mathrm{Diff}_{0}^{\infty}(\mathbb{D}^{2})$ which fixes $\psi^{-1}(K)$ pointwise such that $g=\psi g'\psi^{-1}$. Therefore, $f=g'f'$ and $\xi=\xi'$.\\

For the same reason, the map
$$ \begin{array}{rrcl} 
\Psi': & \mathfrak{diff}_{\mathbb{D}^{2}}(\psi^{-1}(K)) & \rightarrow & \mathfrak{diff}_{S'}(K') \\
 & \xi & \mapsto & \varphi\psi \xi \psi^{-1}\varphi^{-1} \\
\end{array}
$$
is an isomorphism. Hence the map $\Psi' \Psi^{-1}$ is an isomorphism between $\mathfrak{diff}_{S}(K)$ and $\mathfrak{diff}_{S'}(K')$.
\end{proof}

To prove Proposition \ref{representative}, we need the following lemma (see \cite{Hir} Theorem 3.1 p.185).

\begin{lem} \label{diskisotopy}
Let $\Sigma$ be a surface and $e_{1}, e_{2} : \mathbb{D}^{2} \rightarrow \Sigma$ be orientation preserving smooth embeddings such that $e_{1}(\mathbb{D}^{2}) \cap \partial \Sigma = \emptyset$ and $e_{2}(\mathbb{D}^{2}) \cap \partial \Sigma = \emptyset$. Then there exists a diffeomorphism $h$ in $\mathrm{Diff}_{0}^{\infty}(\Sigma)$ such that $h \circ e_{1}=e_{2}$.
\end{lem}

\begin{proof}[Proof of Proposition \ref{representative}]
Let $\xi$ be an element of $\mathfrak{diff}_{S}(K)$ and take a representative $f$ of $\xi$ in $\mathrm{Diff}_{0}^{\infty}(S)$. In the case where the surface $S$ is the sphere, choose a representative $f$ which fixes a point $p$ in $S-D$.\\

Let $\Sigma$ be an embedded compact surface contained in the interior $\mathring{D}$ of $D$ which is a small neighbourhood of $K$. More precisely, this embedded surface $\Sigma$ is chosen close enough to $K$ so that the sets $\Sigma$ and $f(\Sigma)$ are contained in $\mathring{D}$. Observe that this surface $\Sigma$ is not necessarily connected and denote by $\Sigma_{1}, \ldots, \Sigma_{l}$ its connected components. As these surfaces are embedded in a disk, each of these components is diffeomorphic to a disk with or without holes.\\

Fix $1 \leq i\leq l$. Denote by $U_{i}$ the connected component of $S-\Sigma_{i}$ which contains $\partial D$. Finally, let $D_{i}=S-U_{i}$. Observe that the surface $D_{i}$ is diffeomorphic to a disk : it is the surface $\Sigma_{i}$ with "filled holes".

\begin{claim}
For any $i$, $D_{i} \cup f(D_{i}) \subset \mathring{D}$.
\end{claim}

\begin{proof}
In the case where $S \neq \mathbb{S}^{2}$, observe that the connected component of $S-f(\Sigma_{i})$ which contains $\partial D$ (and hence $S-D$ as $f(\Sigma_{i}) \subset D$) is not homeomorphic to a disk. Therefore, this connected component has to be $f(U_{i})$ and $S-\mathring{D} \subset f(U_{i}) \cap U_{i}$. Taking complements, we obtain the desired property.\\

The case of the sphere is similar: $f(U_{i})$ is the only connected component of $S-f(\Sigma)$ which contains $p=f(p)$. 
\end{proof}

Given two disks in the family $D_{1}, D_{2}, \ldots, D_{l}$, observe that either they are disjoint or one of them is contained in the other one. Indeed, for any $i$, the boundary of $D_{i}$ is a boundary component of $\Sigma_{i}$ and the surfaces $\Sigma_{i}$ are pairwise disjoint. Hence, it is possible to find pairwise disjoint disks $D'_{1}, \ldots D'_{m}$ among the disks $D_{1}, \ldots, D_{l}$ such that $D_{1} \cup D_{2} \cup \ldots \cup D_{l}= D'_{1} \cup D'_{2} \cup \ldots \cup D'_{m}$.\\   

We prove by induction on $i$ that, for any $i \leq m$, there exists a diffeomorphism $g_{i}$ supported in $D$ such that
$$f_{|D'_{1} \cup D'_{2} \cup \ldots \cup D'_{i}}=g_{i |D'_{1} \cup D'_{2} \cup \ldots \cup D'_{i}}.$$
Then the diffeomorphism $g_{m}$ is a representative of $\xi$ supported in $D$.\\

In the case $i=1$, use Lemma \ref{diskisotopy} to find a diffeomorphism $g_{1}$ supported in $D$ such that $f_{|D'_{1}}=g_{|D'_{1}}$.\\

Suppose that we have built the diffeomorphism $g_{i}$ for some $i<m$. Observe that the diffeomorphism $g_{i}^{-1}f$ fixes $D'_{1} \cup D'_{2}\cup \ldots \cup D'_{m}$ pointwise and satisfies $g_{i}^{-1}f(D'_{i+1}) \cup D'_{i+1} \subset \mathring{D}$. Apply Lemma \ref{diskisotopy} to find a diffeomorphism $h$ supported in $D-(D'_{1} \cup D'_{2} \cup \ldots \cup D'_{i})$ such that $g_{i}^{-1}f_{|D'_{i+1}}=h_{|D'_{i+1}}$ and take $g_{i+1}=g_{i}h$.
\end{proof}
\end{section}

 \begin{section}{Standard Cantor set}\label{scs}

Fix a parameter $0<\lambda<1/2$. The central ternary Cantor set  $C_{\lambda}$ in $\R\times{\{0\}} \subset {\R}^2 \subset \s^2$  is obtained from the interval $[0,1]$ as the limit of the following inductive process: At the first step,  take out the central subinterval of length $(1- 2\lambda)$ from the interval $[0,1]$ to obtain a collection $\mathcal{I}_1$ of two intervals of size $\lambda$. At the $n$-th step of the process, we obtain a collection $\mathcal{I}_n$ of $2^n$ intervals of size $\lambda^n$ by removing from each interval $I$ in $\mathcal{I}_{n-1}$, the middle subinterval of size $(1-2\lambda)|I|$. Our Cantor set is given by the formula:  $$ C_{\lambda} := \cap_{n \geq 0}(\cup \mathcal{I}_n) $$ 

In this section we will prove Theorem \ref{m2}, which we restate as follows:

\begin{thm}\label{distorted}

Let $S$ be a closed surface. For any $\lambda > 0$, there are no distorted elements in the group $\mcg(S, C_{\lambda})$, where $C_{\lambda}$ is a smooth embedding in $S$ of the standard ternary Cantor set with affine parameter $\lambda$.

\end{thm}

\begin{defn} \label{elementary} We call \textbf{elementary interval} in our Cantor set $C_{\lambda}$ a set of the form $C_{\lambda}\cap I'$, where $I'$ is an interval in the collection $\mathcal{I}_n$ for some $n$. 
\end{defn}

 It was proven in Funar-Neretin (\cite{FN}, see Theorem 6) that any element $\phi \in \mathfrak{diff}_{\R^2}(C_{\lambda})$ is piecewise affine, i.e. there exists a finite collection of elementary intervals $\{I_k\}$ covering $C_{\lambda}$ such that $\phi$ sends $I_k$ into another elementary interval $\phi(I_k)$ and such that $\phi|_{I_k} = \pm \lambda^{n_k}x + c_k$, for some $n_{k} \in \Z$ and an appropriate constant $c_k$.\\
 
 One consequence of this fact is that the group $\mathfrak{diff}_{\R^2}(C_{\lambda})$ does not depend on $\lambda$. Therefore we define the group $\overline{V_2} := \mathfrak{diff}_{\R^2}(C_{\lambda})$. The group $\overline{V_2}$ contains Thompson's group $V_2$, which we define as the group of homeomorphisms of $C_{\lambda}$ which are piecewise affine, where the  affine maps are  of the form $x \to \lambda^{n}x + c$, see \ref{vndef} for another description of $V_2$. (See also \cite{B}, \cite{CFP} and \cite{FN}).\\


 \begin{figure}[ht]
  
  \centering
  
  \vspace{+10pt}
    \includegraphics[width=1.0\textwidth]{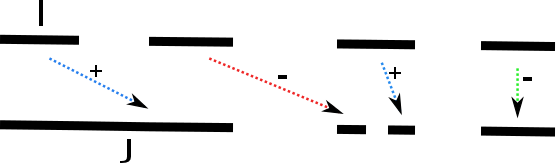}
    \vspace{+3pt}
  \caption{An element $f \in \overline{V_2}$}
    \vspace{+3pt}
    \label{Cantorc}
\end{figure}

\begin{figure}[ht]

  \centering
 \vspace{+10pt}
    \includegraphics[width=1.0\textwidth]{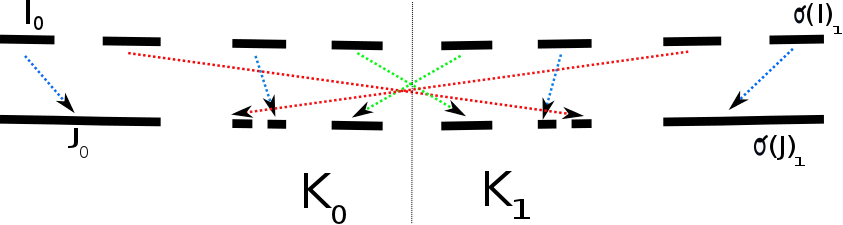}
  \vspace{+3pt}
  \caption{The corresponding  element $\phi(f) \in V_2$.}
  \label{Cantorcc}

     \vspace{+3pt}
\end{figure}

Even though the groups $V_2$ and $\overline{V_2}$ are different and $V_2 \subset \overline{V_2}$, there is an embedding of the group $\overline{V_2}$ into $V_2$ as we will show in the following proposition:\\

 \begin{prop}\label{v2v2} There is an injective homomorphism $\phi: \overline{V_2} \to V_2$.
 \end{prop}
 
 \begin{proof}
 
Let us define  $K := C_{\lambda}$ for some $0<\lambda<1$. Consider the Cantor set $K' := K_0 \cup K_1$, where $K_0 = K$ and $K_1$ is the Cantor set obtained by reflecting $K_0$ in the vertical line $x = 3/2$ as in Figure \ref{Cantorcc}.\\

There is an involution $\sigma: K \rightarrow K$ given by $x \to 1-x$. We will use the following notation: for any interval $I \subset K$, $I_i$ is the corresponding interval of $I$ in the copy $K_i$.\\

For a homeomorphism $f \in \overline{V_2}$ of the ternary Cantor set $C_{\lambda}$, we define $\phi(f) \in V_2$ as the piecewise affine homeomorphism of $K'$ which satisfies the following property. Let $I \subset K$ be any elementary interval such that our homeomorphism $f$ sends $I$ affinely into the elementary interval $J := f(I)$. In this situation, we will define $\phi(f)$ on $I_0 \cup \sigma(I)_1$ by the following rules:\\

\begin{enumerate} 

\item If $f$ preserves the orientation of $I$, we define: $$\phi(f)(I_0) = J_0 \text{ and } \phi(f)(\sigma(I)_1) = \sigma(J)_1.$$
\item If $f$ reverses  the orientation of $I$, we define: $$\phi(f)(I_0) = \sigma(J)_1 \text{ and } \phi(f)(\sigma(I)_1) = J_0.$$
\item In (1) and (2), the maps $\phi(f)|_{I_0}$ and $\phi(f)|_{\sigma(I)_1}$ are affine orientation preserving maps.

\end{enumerate}

One can easily check that $\phi(f)$ is well defined on $K'$ (If  $I \subset J$ and $f|_{I}$ and $f|_J$ are affine maps, then the definition of $\phi(f)$ on $I_0 \cup \sigma(I)_1$ and on $J_0 \cup \sigma(J)_1$ should coincide). See Figure \ref{Cantorcc} for an illustration of this construction.\\

We still need to prove that $\phi$ is a group homomorphism, but that can be shown easily as follows: Suppose $f$ and $g$ are two elements of $\overline{V_2}$. If we partition our Cantor set $K$ into a collection  $\{I_n\}$ of sufficiently small intervals, we can suppose that each interval $I_n$ is mapped affinely by $f$ and also that each interval $f(I_n)$ is mapped affinely by $g$. One can then do a case by case check (whether $f,g$ are preserving orientation or not in $I_n$ and $f(I_n)$ respectively) to show that, restricted to the intervals $(I_n)_0$ and $(\sigma(I_n))_1$ in $K'$, we have $\phi(gf) = \phi(g)\phi(f)$.
 
 \end{proof}

 \textbf{Example:} Consider the element $f$ described in Figure \ref{Cantorc}. The arrows in the picture indicate where $f$ is mapping each of the 4 elementary intervals affinely. The symbols $-,+ $ denote whether the interval is mapped by $f$ preserving orientation or not. The corresponding element $\phi(f)$ is depicted in Figure \ref{Cantorcc}.\\

The previous construction is useful for our purposes because it is known there are no distorted elements in Thompson's group $V_n$ (see Bleak-Collin-et al. \cite{B}, Sec.8, see also Hmili-Liousse \cite{HL}, Corollary 1.10). By Proposition \ref{v2v2}, this implies there are no distorted elements in $\overline{V_2}$. Having that fact in mind and in view of Theorem \ref{main2} we can easily finish the proof of Theorem \ref{distorted}:

\begin{proof}[Proof of Theorem \ref{distorted}] Suppose $f \in \mcg_0(S, C_{\lambda})$ is distorted. Observe that, by Corollary \ref{invariance}, the group $\mathfrak{diff}_{S}(C_{\lambda})$ is isomorphic to $\mathfrak{diff}_{\mathbb{R}^{2}}(C_{\lambda})$. From the exact sequence:\[\pmcg_0(S, C_{\lambda}) \to  \mcg_0(S, C_\lambda)  \underset{\pi}\to  \mathfrak{diff}_{S}(C_{\lambda}) = \overline{V_2}, \]
we obtain that $\pi(f)$ is distorted. Now, by Lemma \ref{v2v2}, $\overline{V_2}$ embeds in $V_2$ and  as there are no infinite order distorted elements in $V_2$ (See \cite{B}, Sec.8), the element $\pi(f)$ has finite order. Thus, we obtained that $f^k \in \pmcg_0(S,  C_{\lambda})$ for some $k\geq 1$, and the element $f^k$ is as well  distorted. By Theorem \ref{main2}, $f^k = \text{Id}$ and so $f$ has finite order.
\end{proof}

\end{section}
 
\begin{section}{Tits alternative}

The ``Tits alternative" states that a finitely generated group $\Gamma$ which is linear (isomorphic to a subgroup of $\text{GL}_n(\R)$ for some $n$) either contains a copy of the free subgroup on two generators $\mathbb{F}_2$ or it is virtually solvable. In \cite{M}, Margulis proved a similar statement for $\text{Homeo}(\s^1)$: Any subgroup $\Gamma \subset \text{Homeo}(\s^1)$ either contains a free subgroup or preserves a measure in $\s^1$. As the derived subgroup $[F_{n},F_{n}]$ of Thompson's group $F_n$ is a simple subgroup of $\text{Homeo}(\s^1)$  (see \cite{CFP}, Theorem 4.5) and does not contain free subgroups on two generators by  Brin-Squier's Theorem (see \cite{G}, Theorem 4.6 p.344), the actual statement of the Tits alternative cannot hold in $\text{Homeo}(\s^1)$ (see also \cite{N}). In this section, we prove Theorem \ref{m3}. By Corollary \ref{invariance}, Theorem \ref{m3} reduces to the following theorem.

\begin{thm}\label{main}
Let $\Gamma$ be a finitely generated subgroup of $\mcg(\mathbb{R}^{2}, C_{\alpha})$, then one of the following holds:

\begin{enumerate}

\item $\Gamma$ contains a free subgroup on two generators $\mathbb{F}_2$
\item $\Gamma$ has a finite orbit, i.e. there exists $p \in C_{\alpha}$ such that the set $\Gamma(p) := \{ g(p) \: | g \in \Gamma\}$ is finite.

\end{enumerate}

\end{thm}

Using the description of $\mathfrak{diff}_{\mathbb{R}^{2}}(C_{\lambda})$ explained at the beginning of Section \ref{scs} and Proposition \ref{v2v2}, we deduce the previous theorem as an immediate corollary of the following statement about  Thompson's group $V_n$, which could be of independent interest:\\

\begin{thm} For any finitely generated subgroup  $\Gamma \subset V_n$, either the action of $\Gamma$ on the  Cantor set $K_n$ has a finite orbit or $\Gamma$ contains a free subgroup.
\end{thm}

The finite generation condition is indeed necessary, as the following example shows: \\

The finite group of permutations $S_{2^n}$ is a subgroup of $V_2$ as it acts on $C_{\lambda}$ by permutations of the elementary intervals of the collection $\mathcal{I}_n$ described at the beginning of Section \ref{scs}. Defining the group $S_{\infty} := \bigcup_n S_{2^n}$, we easily see that $S_{\infty}$ has no finite orbit (the action is in fact minimal) and  there is no free subgroups in $S_{\infty}$ as any element has finite order.\\

\begin{subsection}{Elements of $V_n$ and tree pair diagrams}\label{vndef}

We need to describe the action of the elements of $V_n$ on the Cantor set $K_n$ in detail. In order to do this, we will use as a tool the description of the elements of $V_n$ as tree pair diagrams as described for example in \cite{B}, \cite{Br}, \cite{CFP}, \cite{SD}. We will follow very closely the description given in \cite{B} and we refer to it for a more detailed explanation of the material introduced in this subsection. The main tool for us is the existence of ``revealing tree pair diagrams" which were first introduced by Brin in \cite{Br}. These ``revealing diagrams" allow us to read the dynamics of elements of $V_n$ easily.\\

Indeed, we will show that for each element $g \in V_n$, there are two $g-$invariant clopen sets $V_g$ and $U_g$ such that $K_n = U_g \cup V_g$, where $g|_{U_g}$ has finite order and $g|_{V_g}$ has ``repelling-contracting" dynamics. The reader that decide to skip this introductory subsection, should look at Lemma \ref{contra}, where all the properties of the dynamics of elements of $V_n$ in $K_n$ that we will use are described.\\

\begin{subsubsection}{Notation}

 \begin{figure}[ht]

  \centering
    \vspace{+10pt}
    \includegraphics[width=1.1\textwidth]{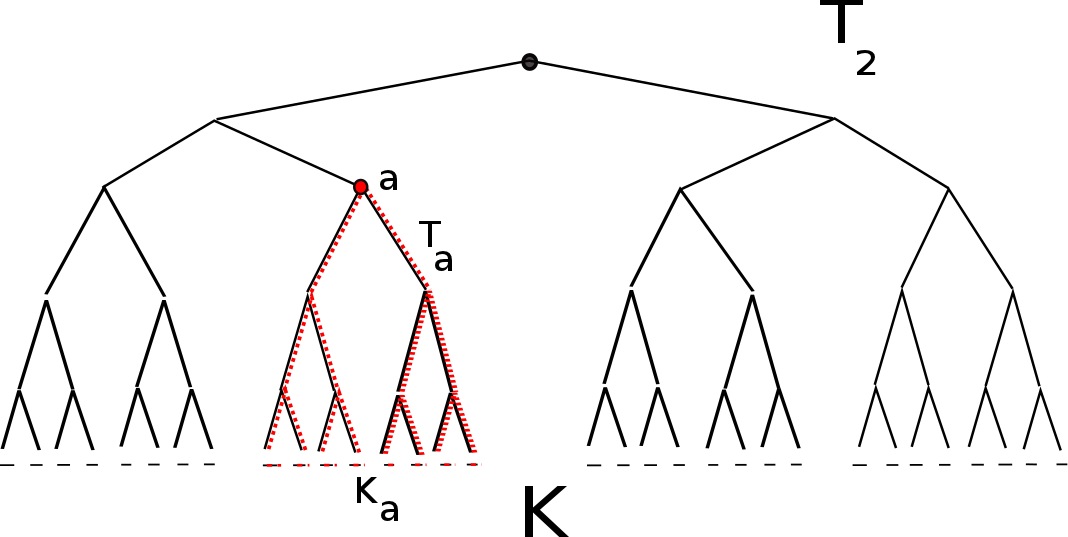}
    \vspace{+3pt}
  \caption{The rooted tree $\mathcal{T}_2$, for a vertex $a \in \mathcal{T}_2$, the tree $\mathcal{T}_a$ and the elementary interval $K_{a}$ are depicted}
    \vspace{+3pt}
    
   \label{t2}
\end{figure}

From now on, $K_n$ denotes the Cantor set that we identify with the ends of the infinite rooted $n-$tree $\mathcal{T}_n$ (see Figure \ref{t2}, where $\mathcal{T}_2$ and $K$ are depicted). For a vertex $a \in \mathcal{T}_n$, we define $T_a$ as the infinite $n$-ary rooted tree descending from $a$. We define the clopen set $K_a \subset K$ as the ends of $T_{a}$ (see Figure \ref{t2}). Any subset of $K_n$ of the form $K_a$ is an \textbf{elementary interval} as in Definition \ref{elementary}.\\

\end{subsubsection}

The group $V_n$ is a subgroup of the group $\text{Homeo}(K_n)$. Our next task is to describe which kind of homeomorphisms of $K_n$ belong to $V_n$, looking at the example depicted in Figure \ref{vnrota2} might be instructive to understand what an element of $V_n$ can be.\\

 \begin{figure}[ht]
  
  \centering
  
  \vspace{+10pt}
    \includegraphics[width=1.1\textwidth]{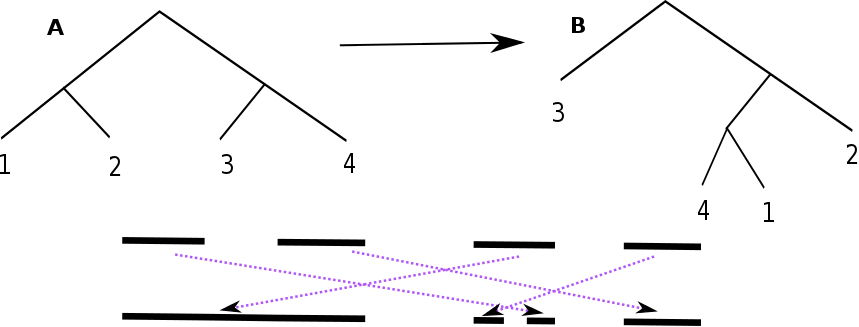}
    \vspace{+3pt}
  \caption{An element $f$ of $V_2$}
    \vspace{+3pt}
    \label{vnrota2}
    
\end{figure}

An element $g$ of $V_n$ is described by a triple $(A,B, \sigma)$ where $A$ and $B$ are $n$-ary rooted trees (connected subtrees of $\mathcal{T}_n$) with the same number of endpoints ($A,B$ tell us a way of partitioning our Cantor set $K_n$ into elementary intervals)  together with a bijection $\sigma$ between the endpoints of $A$ and the endpoints of $B$ that  tell us how an interval is going to be mapped by $g$ to another interval. More formally, for an endpoint $a \in A$, $g$ maps $K_a$ into $K_{\sigma(a)}$ by mapping $T_{a}$ into $T_{\sigma(a)}$ in the obvious way (see Figure \ref{vnrota2}).\\

\end{subsection}

\begin{subsection}{Revealing pairs}

 \begin{figure}[ht]
  
  \centering
  
  \vspace{+10pt}
    \includegraphics[width=1.1\textwidth]{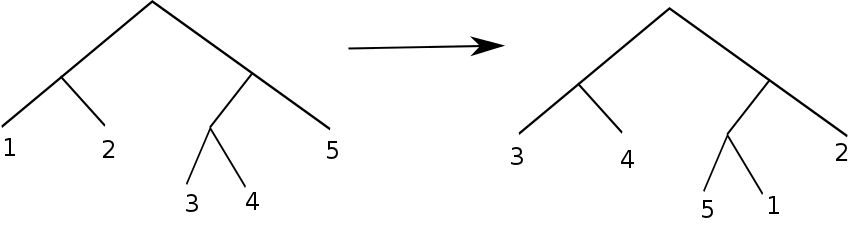}
    \vspace{+3pt}
  \caption{A revealing pair diagram for the element in Figure \ref{vnrota2}.}
    \vspace{+3pt}
    \label{vnrota4}
    
\end{figure}

It should be noted that an element $g \in V_n$ is not described by a unique tree pair diagram $(A,B, \sigma)$. Some tree pair diagrams describe the dynamics of an element $V_n$ better than others. As an example, consider the element $f$ with the diagrams depicted in Figures \ref{vnrota2} and \ref{vnrota4}. In the tree pair diagram in Figure \ref{vnrota4}, the trees $A$ and $B$ coincide and so $f$ must have finite order.\\ 

Before defining what a revealing tree pair diagram is, we will need to set up some notation. Let $g \in V_n$ be an element described by a tree pair $(A,B, \sigma)$.\\

Consider the sets $X = \overline{A-B}$ and $Y = \overline{B-A}$. Notice that each connected component of  $X$ (respectively $Y$) is a rooted tree whose root is an endpoint of $B$ but not of $A$ (respectively $A$ but not $B$). In the example depicted in Figure \ref{vngood}, $X$ is blue and $Y$ is red.\\

Let $L_{(A,B, \sigma)}$ denote the set of vertices of $\mathcal{T}_n$ which are  endpoints of either $A$ or $B$. A vertex in $L_{(A,B, \sigma)}$ is called neutral if $x$ is an endpoint of both $A$ and $B$. Observe that if $\lambda \in L_{(A,B, \sigma)}$ and if  $g^i(\lambda)$ is a neutral vertex of $L_{(A,B, \sigma)}$ for every $i \geq 0$, then the vertex $\lambda$ must be periodic for $g$. Let $t$ be the period of $\lambda$, \emph{i.e.} the minimal $t>0$ such that $g^{t}(\lambda) = \lambda$. Observe that in this case $g^t|_{K_\lambda} = \text{Id}$. If $\lambda$ is not periodic  we can find the largest integers  $s\geq 0$ and $r \geq 0$, such that for any $-r<i< s $, the vertex  $g^i(\lambda)$ is a neutral vertex of $L_{(A,B, \sigma)}$. In this case, we define the iterated augmentation chain as $$IAC(\lambda) := \big(g^i(\lambda)\big)_{i=-r}^s .$$ Observe that the vertex $g^{-r}(\lambda)$ is an endpoint of $A$ but not of $B$ and the vertex  $g^{s}(\lambda)$ is an endpoint of $B$ but not of A.\\

An \emph{attractor} in $L_{(A,B, \sigma)}$ is defined as an endpoint of $A$ such that $g^{s}(\lambda)$ belongs to $B \setminus A$ and such that $g^{s}(\lambda)$ is strictly contained in $T_\lambda$ ($g^s(\lambda)$ is under $\lambda$). In this case we see that $g^{s}|_{K_\lambda}$ has attracting dynamics, and there is a unique attracting point $p$ for $g^s$ inside $K_\lambda$. In a similar way, one defines a ``\emph{repeller}" as a vertex $\lambda$ in $B$, such that $g^{-r}(\lambda)$ is strictly contained in $T_{\lambda}$. In Figure \ref{vngood}, the red vertices are attractors and the blue one is a repeller. Observe that attractors are always roots of components of $Y$ and repellers are always roots of components of $X$. \\


\begin{defn} Let $(A, B,  \sigma)$ be a tree pair diagram for an element $g \in V_n$. The set $X = \overline{A\setminus B}$  (respectively $Y = \overline{B \setminus A}$) consists of a union of rooted trees, whose roots are endpoints of $B$ (respectively $A$). If all these vertices are repellers (respectively attractors), then $(A, B,  \sigma)$ is said to be a \emph{revealing tree pair diagram}.
\end{defn}

\begin{thm}[Brin \cite{Br}]\label{rp} For every $g \in V_n$, there exists a revealing tree pair diagram $(A,B, \sigma)$, even more,  there is an algorithm to extend any tree pair diagram into a revealing tree pair diagram.

\end{thm}

One easy consequence of Theorem \ref{rp} is that every periodic element of $V_n$ has a tree pair diagram $(A, B, \sigma)$ where $A = B$, as it is illustrated in Figure \ref{vnrota4}.\\

If $(A, B, \sigma)$ is a revealing pair, the dynamics of each interval under a vertex of $L_{(A, B, \sigma)}$ can be easily described as we will show next.\\

Let $\lambda$ be a vertex of $L_{(A,B,\sigma)}$ and suppose that $\lambda$ is not a periodic vertex. Let  $ \text{IAC}(\lambda) = \big(g^i(\lambda)\big)_{i=-r}^s$ be its iterated augmented chain.  In this case, $g^{-r}(\lambda)$ is an endpoint of $A$ but not of $B$, and $g^{s}(\lambda)$ is an endpoint of $B$ but not of $A$. Hence, there are two possibilities: either $g^{s}(\lambda)$ is a root of a component of $X$, or $g^{s}(\lambda)$ is a vertex of a tree in $Y$.\\  

If $g^{s}(\lambda)$ is a root of a component of $X$, then, as $(A, B, \sigma)$ is a revealing tree pair diagram, $g^{s}(\lambda)$ is a repeller. The vertex $g^{-r}(\lambda)$ is then strictly under $g^s(\lambda)$ and there is a unique fixed point $p$ for $g^{-r-s}$ in $K_{g^{s}(\alpha)}$. This point $p$ is a repelling periodic point of order $s+r$. In that case, the elementary intervals $\{g^i(K_{\lambda})\}_{i=-r}^{s-1}$ are disjoint and each of them contains a unique repelling periodic point in the orbit of $p$.\\  

If $g^{s}({\lambda})$ is a vertex of a tree in $Y$ and  the vertex $g^{-r}(\lambda)$ is a root of a component of $Y$, then $g^{-r}(\lambda)$ is an attractor. In this case, there is a unique periodic attracting point $q$ of order $s+r$ in $K_{g^{-r}(\alpha)}$. Again, the intervals $\{g^i(K_{\lambda})\}_{i=-r}^{s-1}$ are disjoint and each of them contains a unique attracting periodic point in the orbit of $p$.\\

If the vertex $g^{-r}({\lambda})$ is a vertex of a tree in $Y$ under a repeller $\alpha$ and  the vertex $g^{s}(\lambda)$  is a vertex of a tree in $X$ under an attarctor $\omega$, then we see that the forward orbit of $K_{\lambda}$ is getting attracted toward the periodic orbit $p_{\omega}$ corresponding to $\omega$ and the backward orbit of $K_{\lambda}$ gets attracted toward the periodic repeller $p_{\alpha}$ corresponding to $\alpha$. \\

 \begin{figure}[ht]
  
  \centering
 
  \vspace{+10pt}
    \includegraphics[width=1.1\textwidth]{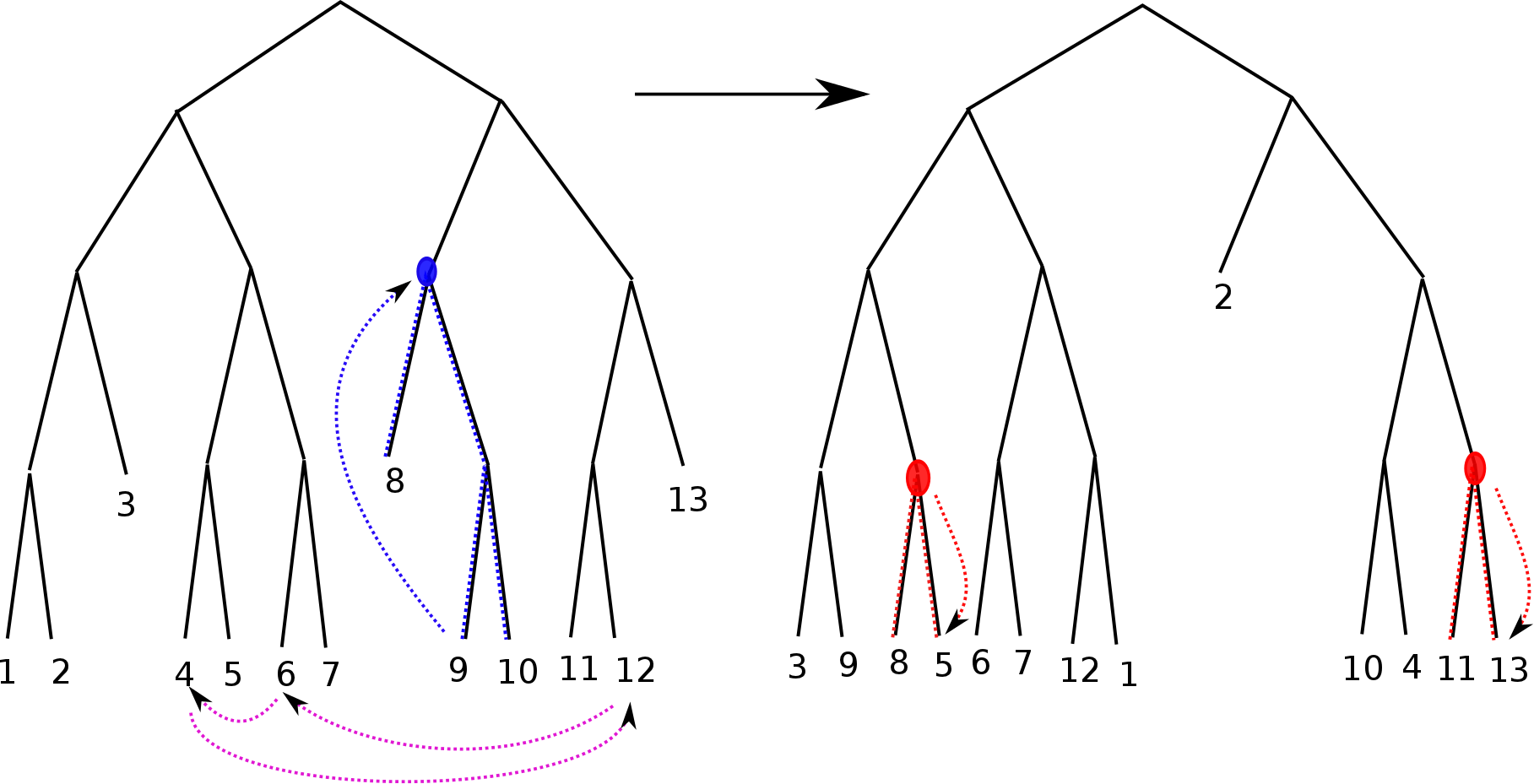}
     \vspace{+3pt}
  \caption{A revealing pair diagram for an element of $V_2$}
    \vspace{+3pt}
    \label{vngood}
\end{figure}

\textbf{Example} (see Figure \ref{vngood}). We use the following notation, for a number $j$, we denote by $j_A$ the vertex in $\mathcal{T}_2$ numbered by $j$ in the tree $A$. We define similarly the vertices $j_B$. Observe that for our particular example we have $1_A = 3_B$ and $2_A = 9_B$.\\

In Figure \ref{vngood}, $X$ is depicted blue. It consists of one rooted tree with root $2_B$. Observe that, for this element, we have $$9_A \to 9_B = 2_A \to 2_B$$ and the vertex $9_A$ is under $2_B$, and so $2_B$ is a repeller. There is a unique repelling periodic point under $2_B$ of period $2$.\\

The set $Y$ consists of two trees, one tree with root $3_A$ and the other one with root $13_A$, we have $$3_A \to 1_A \to 7_A \to 5_A \to 5_B$$ and $5_B$ is under $3_A$ and so $3_A$ is an attractor, there is an attracting periodic point of period $4$ under $3_A$ . We also notice that $13_B$ is under $13_A$ and so there is an attracting fixed point under $13_A$.\\

Observe also that  $10_A \to 11_A \to 11_B $, the vertex $10_A$ is under the repeller $2_B$ and $11_B$ under the attractor $13_A$, this means there are points arbitrarily close to the repelling periodic point that converge toward the attracting fixed point under $13_A$ and we also have $8_A \to 8_B$,  where $8_A$ is under the repeller $2_B$ and $8_B$ under the attractor at $3_A$, and so there are orbits going from the repelling periodic orbit to the attracting periodic orbit corresponding to $3_A$.\\
 
As a consequence of the previous discussion, we obtain the following lemma. For a more detailed discussion, see \cite{B} and \cite{SD}.\\

\begin{lem}\label{reveal}

Let $(A,B, \sigma)$ be a revealing pair for an element $g \in V_n$. Let $\lambda$ be a vertex in $L_{A,B, \sigma}$, then exactly one of the following holds.

\begin{enumerate}

\item $\lambda$ is periodic, in which case there is $t>0$ such that  $g^{t}{\lambda} = \lambda$ and $g^t|_{K_{\lambda}} = \text{Id}$.
\item $K_\lambda$ contains a unique contracting periodic point $p$ and there is $t > 0$ such that $g^{t}(K_\lambda) \subset K_\lambda$  and $g^{t}|_{K_{\lambda}}$ is contracting affinely (i.e $g^t$ sends the interval $K_{\lambda}$ into the interval $g^t(K_{\lambda})$ in the obvious way).
\item $K_\lambda$ contains a unique repelling periodic point $p$ in $\lambda$ and there is $r > 0$ such that $g^{-r}(K_\lambda) \subset K_\lambda $  and $g^{-r}|_ {K_{\lambda}}$ is contracting affinely.
\item There exist $s \geq 0, r \geq 0$ such that $g^s(\lambda)$ and $g^{r}(\lambda)$  are vertices but not roots of components of $Y$ and $X$ respectively. In this case, the following property holds. As  $n \to \infty$,   $g^n(K_\lambda)$ gets contracted affinely converging towards an attracting periodic orbit of $g$  and $g^{-n}(K_{\lambda})$ gets attracted towards a repelling periodic orbit of $g$. 

\end{enumerate}

\end{lem}

In the proof of Theorem \ref{main}, we use the following notation which makes the proof easier to digest.\\

\begin{defn}[Neighborhoods] In the Cantor set $K_n$, let us consider the metric coming from the standard embedding of $K_n$ in the interval $[0,1]$. For a point $p \in K_n$ and $\epsilon>0$, we define the neighborhood $N_{\epsilon}(p)$ as the  maximal elementary interval $I$ such that $p \in I$ and $\text{length}(I) < \e$. Similarly, if $S$ is a finite set, we define $N_{\e}(S) := \cup_{s \in S}N_{\e}(s)$.
\end{defn}

\begin{defn} For any element $g \in V_n$, we define the following: $$\atr(g):= \{\:p \in K_n \: \text{  such that } p \text{ is periodic and attracting}  \}$$ $$\rep(g):= \{\:p \in K_n \: \text {  such that } p \text{ is periodic and repelling}  \}$$ $$\per(g) := \atr(g) \cup \rep (g)$$ By Lemma \ref{reveal}, both sets $\atr(g)$ and $\rep(g)$ are finite. Hence the set $\per(g)$ is also finite. 

\end{defn}

As a conclusion of Lemma \ref{reveal} we obtain the following lemma that enumerates all the dynamical properties of the action of elements $V_n$ on $K_n$ that we will use.\\

\begin{lem}\label{contra}

Given an element $g \in V_n$ there exist two $g-$invariant clopen sets $U_g, V_g$ (i.e. finite union of  elementary intervals in $K_n$) such that:

\begin{enumerate}

\item $K_n = U_g \cup V_g$.\\ 
\item $g|_{U_g}$ has finite order.\\
\item There are finitely many periodic points of $g$ contained in $V_g$( the set $\per(g)$) and the dynamics of $g|_{V_g}$ are ``attracting-repelling" i.e.  for every $\e >0$, there exists $m_0$  such that, for $m \geq m_0$, we have: $$g^m(V_g \setminus N_\e(\rep(g))) \subset N_{\e}(\atr(g))$$   $$g^{-m}(V_g \setminus N_\e(\atr(g))) \subset N_{\e}(\rep(g)).$$
\item  If $\epsilon$ is small enough, for any point $p \in \atr(g)$, there exists $s$ (the period of $p$) such that $g^s(N_{\epsilon}(p)) \subset N_{\e}(p)$ and $g^s|_{N_{\e}(p)}$ is an affine contraction. The analogous condition also holds for points in $\rep(g)$.
\end{enumerate}

\end{lem}

\end{subsection}

\begin{subsection}{Proof of Theorem \ref{main}}

The idea of the proof of Theorem \ref{main} is to use the ``attracting-repelling" dynamics of elements of $V_n$ and  the ``ping-pong" lemma to obtain a free subgroup $\mathbb{F}_2$ contained in $\Gamma$ (this strategy was the one used by Margulis to prove his ``alternative" for $\text{Homeo}(\s^1)$, see \cite{M}). To illustrate the idea of the proof, suppose our group $G$ contains two elements $f$ and $h$ such that $h$ sends $\per(f) \cup U_f$ disjoint from itself. In that case, if we consider the element $g = hfh^{-1}$, the sets $\per(f) \cup U_f$ and $\per(g) \cup U_g$ are disjoint.\\

Under this last condition, one can apply the ping-pong lemma (see \cite{H}, Ch. 2) as follows to show that $\langle f^n,g^n \rangle$ generate a free group if $n$ is large enough. Let us take small disjoint neighborhoods $N_\e(\per(f))$ and $N_{\e}(\per(g))$  and consider the set $$X = K \setminus (U_f \cup U_g \cup N_\e(\per(f)) \cup N_{\e}(\per(g)) ).$$ If we take $\epsilon$ small enough, then $X \neq \emptyset$, and, by Lemma \ref{contra}, if $n$ is large enough, we have $$f^{n}(X) \subset N_\e(\atr(f)) $$ $$f^{-n}(X) \subset N_\e(\rep(f)) $$  and we also have $$f^{n}(N_\e(\per(g))) \subset N_\e(\atr(f))$$ $$f^{-n}(N_\e(\per(g))) \subset N_\e(\rep(f)).$$ The corresponding statement for $g^n$ and $g^{-n}$ are also true. This implies that $f^n, g^n$ generate a free group, as for any nontrivial word $w$ on the elements $f^n, g^n$, we have by the ``ping-pong" argument that $w(X) \subset N_{\epsilon}(\per(f) \cup \per (g))$  and therefore $w \neq \text{Id}$.\\ 

As a conclusion, we have proved the following lemma:

\begin{lem}\label{pingpong} Let $\Gamma \subset V_n$. If there are two elements $f, h \in \Gamma$ such that the sets $\per(f) \cup U_f$ and  $h(\per(f) \cup U_f)$ are disjoint, then $\Gamma$ contains a free subgroup on two generators.

\end{lem}

To prove Theorem \ref{main}, we will show that either a pair of elements $f,h$ of $\Gamma$ as in Lemma \ref{pingpong} exists or that $\Gamma$ has a finite orbit in $K_n$.  The following result is the key lemma for proving the existence of such an element $h$ sending $\per(f) \cup U_f$ disjoint from itself. It is based on a recent proof by Camille Horbez (See \cite{CH1}, Sec.3, \cite{CH2}) of the Tits alternative for mapping class groups, outer automorphisms of free groups and other related groups.\\

\begin{lem}\label{period} Let $\Gamma$ be a countable group acting on  a compact space $K$ by homeomorphisms and let $F \subset K$ be a finite subset. Then either there is finite orbit of $\Gamma$ on $K$ or there exists an element $g\in \Gamma$ sending $F$ disjoint from itself (i.e. $g(F ) \cap F =  \emptyset $).

\end{lem}

Before beginning the proof of Lemma \ref{period}, we recall the following basic notions of random walks on groups and harmonic measures.\\

For a discrete group $\Gamma$, let us take a probability measure $\mu$ on $\Gamma$ and suppose that $\langle \text{supp}(\mu) \rangle = \Gamma$.  Suppose our group $\Gamma$ acts continuously on a compact space $X$. A harmonic measure in $X$ for $(\Gamma, \mu)$ is a Borel probability measure $\nu$ on $X$ such that $\mu * \nu = \nu$, where "$*$" denotes the convolution operator. This means that, for every $\nu$-measurable set $A\subseteq X$,
\begin{equation}
\nu(A) = \sum_{g \in \Gamma} \nu(g^{-1}(A))\mu(g) 
\end{equation}

A harmonic measure always exists (see the proof below)  and one can think of it as a measure on $X$ that is invariant under the action of $\Gamma$ on average (with respect to $\mu$).\\

\begin{proof}[Proof of Lemma \ref{period}] Suppose that there is no element of $\Gamma$ sending $F$ disjoint from itself. Let $n = |F|$. If $n=1$, the theorem is obvious and so we assume $n > 1$. Consider the diagonal action of $\Gamma$ on $K^n$. Let  $\vec{p} = (p_1,p_2,...,p_n)$ be an $n$-tuple consisting of the $n$ different elements of $F$ in some order. We take a probability measure $\mu$ supported in our group $\Gamma$ such that $\langle \text{supp}{\mu} \rangle = \Gamma$ and take a harmonic probability measure $\nu$ on $K^n$ supported in $\overline{\Gamma \vec{p}}$. Such a harmonic measure $\nu$ can be obtained as follows: Take the Dirac probability measure $\delta_{\vec{p}}$ in $K^n$ supported in $\left\{ \vec{p} \right\}$ and consider the averages of convolutions $\nu_l :=\frac{1}{l} \sum^l_{i=1} \mu^i * \delta_{\vec{p}}$ ($\mu^i$ is the measure obtained by convoluting $\mu$ $i$ times with itself). Then $\nu$ can be taken as any accumulation point of $\nu_l$ in the space of probability measures in $K^n$.\\

Observe that, by our assumption, for each $g \in \Gamma$,  the element $g(\vec{p})$ is contained in a set of the form $K^l\times{\{p_i\}}\times K^m$ for some integers $i$, $l$ and $m$ such that $l+m=n-1$ and therefore: $$\overline{\Gamma \vec{p}} \subset \bigcup_{0 \leq i \leq n,\text{ } l+m={n-1}} K^l\times{\{p_i\}}\times K^m.$$ As $\nu(\overline{\Gamma \vec{p}}) = 1$, we can conclude that there exist integers $i$, $l$ and $m$ such that $\nu (K^l\times{\{p_i\}}\times K^m) > 0$.\\ 

Let us take $q \in K$ such that $\nu (K^l\times{\{q\}}\times K^m)$ is maximal. We will show that $q$ has a finite $\Gamma$-orbit. Observe that, for $g \in \Gamma$, $g(K^l\times{\{q\}}\times K^m) = K^l\times{\{g(q)\}}\times K^m$ and, therefore,
$$\nu (K^l\times{\{q\}}\times K^m) = \sum_i \nu (K^l\times{\{g_i^{-1}(q)\}}\times K^m)  \mu(g_i).$$ So we obtain by our maximality assumption that $\nu (K^l\times{\{q\}}\times K^m) = \nu (K^l\times{\{g^{-1}(q)\}}\times K^m)$  for every $g$ in the support of $\mu$. Hence this also holds for every $g \in \Gamma$. But being $\nu$ a probability measure this can only happen if the orbit $\Gamma(q)$ is finite and so we are done.
\end{proof}

\begin{rem} One can also conclude with a little bit more of extra work that, for some $i$, $\Gamma(p_i)$ is a finite orbit. We will not make use of this fact.

\end{rem}

The following proposition is our main tool to construct free subgroups of a subgroup $\Gamma$ of $V_n$.

\begin{prop}\label{rf} Suppose $f,g \in \Gamma \subset V_n$ are such that $U_f$ and $U_g$ are disjoint and suppose there is no periodic orbit for $\Gamma$ in $K_n$. Then, there exists a free subgroup on two generators contained in $\Gamma$.\\
\end{prop}

\begin{proof}

Taking powers of $f$ and $g$ we can suppose that $f|_{U_f} = \text{Id}$,  $g|_{U_g} = \text{Id}$ and that the repelling and attracting periodic points of $f$ and  $g$ are fixed by $f$ and $g$ respectively. We will prove that, given any $\epsilon >0$, there exists an element $w_{\e} \in G$ such that $\per(w_{\e}) \cup U_{w_{\e}}$ is contained in $N_\e(\per(f) \cup\per(g))$. First, let us show this implies Proposition \ref{rf}.\\

By Lemma \ref{period}, we can find an element $h \in \Gamma$ sending $\per(f)\cup \per(g)$ disjoint from itself. Hence, if $\epsilon$ is small enough, $h$ sends $N_\e(\per(f)) \cup N_{\e}(\per(g))$ disjoint from itself, which implies that the sets $\per(w_{\e}) \cup U_{w_{\e}}$ and $h(\per(w_{\e}) \cup U_{w_{\e}})$ are disjoint. By Lemma \ref{pingpong}, $\Gamma$ contains a free subgroup.\\

We will prove that our desired element $w_{\epsilon}$ can be taken of the form $w_{\e} := g^{m_1}f^{m_2}$. To illustrate the idea of the proof of this fact, suppose first that $\per(f)$ and $\per(g)$ are disjoint. In this case, let us define the set $ V := V_f - N_{\e}(\rep(f))$. Take $\e$ small enough so that, for a point $p \in \atr(f)$ either  $N_\e(p) \subset U_g$ or $N_{\e}(p)$ is contained in $V_g \setminus N_\e(\rep(g))$. Also, take $\e$ small enough so that the sets $N_\e(\atr(f))$, $N_\e(\rep(f))$, $N_\e(\atr(g))$ and $N_\e(\rep(g))$ are pairwise disjoint. By Lemma \ref{contra}, we can take $m$ large enough so that $f^m(V) \subset N_{\epsilon}(\atr(f))$ and $g^m(N_\e(\atr(f))) \subset  N_\e(\atr(f) \cup \atr(g))$. As a conclusion, we obtain that: 
\begin{equation}\label{1}
g^mf^m(V) \subset N_\e(\atr(f)) \cup N_\e(\atr(g)).
\end{equation}

Also, if we consider the set $U := U_f \setminus N_\e(\rep(g))$, taking $m$ larger if necessary, we have

 \begin{equation}\label{2}
g^mf^m(U) = g^m(U) \subset g^m(V_g  \setminus N_\e(\rep(g))) \subset N_\e(\atr(g)).
\end{equation}

As we are assuming for the moment that the sets $\per(f)$ and $\per(g)$ are disjoint, we also obtain that 

\begin{equation}\label{3}
g^mf^m(N_\e(\atr(f)) \cup N_\e(\atr(g))) \subset N_\e(\atr(f)) \cup N_\e(\atr(g))
\end{equation}

Inclusions \ref{1}, \ref{2} and \ref{3} imply that, for $\e$ small enough and $m$ sufficiently large, all the periodic points of $ w_{\e}:= g^mf^m$ in $K \setminus N_{\e}(\rep(f) \cup \rep(g)) \subset U \cup V $ must be contained in $ N_\e(\atr(f)) \cup N_\e(\atr(g))$ and therefore  the periodic points of $w_{\e}$ must be contained in $N_\e(\per(f)) \cup N_{\e}(\per(g))$ as we wanted.\\

 To finish the proof of Proposition \ref{rf}, we need to deal with the case where $\per(f)$ and $\per(g)$ have points in common. This case is significantly trickier but the proof is similar to the one above. We include this case as an independent lemma:\\

\begin{lem}\label{disjoint} Let $f$, $g$ be elements of $V_n$ such that $U_f$ and $U_g$ are disjoint. For every $\epsilon > 0$, we can find an element $w_{\e} \in V_n$ of the form $w_{\e} = g^{m_1}f^{m_2}$ such that all the periodic points of $w_{\e}$ (\emph{i.e.} $U_{w_{\e}} \cup \per(w_{\e})$) are contained in $N_{\e}(\per(f) \cup \per(g))$.\\

 \end{lem}
 
\begin{proof}

We can suppose all the periodic points of $f$ and $g$ are fixed and, taking $\e>0$ small enough, we can suppose that  for $p \in \per(f)\cup \per(g)$, the sets $N_{\e}(p)$ are pairwise disjoint  and entirely contained in the sets $U_f, V_f, U_g, V_g$ if $p$ intersects such a set.\\

Let $\e_0:= \e$ and take an integer $n$ large enough such that
$$ g^{n}(V_{g} \setminus N_{\epsilon_{0}}(\rep(g)))\subset N_{\epsilon_{0}}(\atr(g)).$$
Choose $0<\e_{1}<\e_{0}$ small enough so that
$$N_{\e_{1}}(\atr(g) \cap \rep(f)) \subset g^{n}(N_{\epsilon_{0}}(\atr(g) \cap \rep(f)))$$
and
$$g^{n}(N_{\e_{1}}(\rep(g)\cap \atr(f)))\subset N_{\epsilon_{0}}(\rep(g) \cap \atr(f)).$$
Finally choose an integer $m$ large enough so that
$$f^{m}(V_{f} \setminus N_{\epsilon_{1}}(\rep(f))) \subset N_{\epsilon_{1}}(\atr(f)).$$

We can now define the sets $$W_g:=V_g \setminus N_{\e_0}(\rep(g) \cup  (\atr(g)\cap \rep(f))) $$ and $$A_{\e_0,\e_1} := N_{\e_0} (\atr(g)\cap \rep(f)) \setminus N_{\e_1} (\atr(g) \cap \rep(f)).$$

We observe that by our choices of $\e_0,\e_1$ and $n$, we have: 
\begin{equation}\label{4}
g^n(W_g)  \subset N_{\e_0} (\atr(g) \setminus \rep(f)) \cup  A_{\e_0,\e_1}.
\end{equation}

We will show that the element $w_{\epsilon} :=g^nf^m$ has the desired properties. One should have in mind that $m$ is chosen much bigger than $n$ in order to guarantee that all the  points in $\atr(f)\cap \rep(g)$ are attractors for $w_{\e}$.\\
 
We define the set: 
$$X := N_{\e_0}(\atr(f)) \cup N_{\e_0} (\atr(g)\setminus\rep(f)) \cup A_{\e_0,\e_1}.$$
First, we show that $X$ is attracting most of $K$ towards itself. More concretely, we show the following:

\begin{lem}\label{hard}

For the set $X$ defined above, the following properties hold:

\begin{enumerate}
\item $g^nf^m(X) \subset X$ (Invariance)
\item $g^nf^m(K \setminus N_{\e_0}(\rep(f) \cup \rep(g))) \subset X$. (Contractivity)

\end{enumerate}
\end{lem}

\begin{proof}
We start by proving item (1). $X$ was defined as the union of the sets $N_{\e_0}(\atr(f)),  N_{\e_0} (\atr(g) \setminus \rep(f))$ and $A_{\e_0,\e_1}$, we will show that when we apply $g^nf^m$ to each of these sets, the resulting set is still contained in $X$. Let us start with $N_{\e_0}(\atr(f))$. We have:

$$ g^nf^m(N_{\e_0}(\atr(f))) \subset g^n(N_{\e_1}(\atr(f))).$$

To understand $g^n(N_{\e_1}(\atr(f)))$, we consider each of the cases whether $\atr(f)$ intersects the sets $\per(g)$, $U_g$ or none of them. Observe the following:\\

\begin{itemize}

\item By our choice of $\e_1$, we have: $$g^n(N_{\e_1}(\atr(f)\cap \rep(g))) \subset N_{\e_0}(\atr(f)\cap \rep(g)) \subset X.$$
\item As $g$ is attracting in $N_{\e_1}(\atr(g))$, we have: $$g^n(N_{\e_1}(\atr(f) \cap \atr(g))) \subset N_{\e_1}(\atr(f) \cap \atr(g)) \subset X.$$ 
\item As $g|_{U_g} = \text{Id}$, we have: $$g^n(N_{\e_1}(\atr(f)) \cap U_g) = N_{\e_1}(\atr(f)) \cap U_g \subset X$$ 
\item  As $V_g \setminus N_{\epsilon}(\per(g)) \subset W_g$ and by Inclusion \ref{4}  we know that  $g^n(W_g) \subset N_{\e_0} (\atr(g) \setminus \rep(f)) \cup  A_{\e_0,\e_1} \subset X$, we have:
 $$g^n(N_{\e_1}(\atr(f)) \cap (V_g \setminus N_{\epsilon}(\per(g)))) \subset g^n(W_g) \subset X.$$
\end{itemize}

As s consequence, we obtain that $g^n(N_{\e_1}(\atr(f))) \subset X$ and therefore $g^nf^m(N_{\e_0}(\atr(f))) \subset X$ as we wanted.\\

We now consider the set  $N_{\e_0} (\atr(g) \setminus \rep(f))$. We distinguish two cases, whether $N_{\e_0} (\atr(g) \setminus \rep(f))$ intersects $U_f$, or $V_f$. Let us consider the former case first. Observe that:
\begin{equation}\label{6}
f^m(N_{\e_0} (\atr(g) \setminus \rep(f)) \cap U_f) = N_{\e_0} (\atr(g) \setminus \rep(f)) \cap U_f.
\end{equation}

We clearly have: 
\begin{equation}\label{7}
g^n(N_{\e_0} (\atr(g) \setminus \rep(f))) \subset N_{\e_0} (\atr(g) \setminus \rep(f)).
\end{equation} From Inclusions \ref{6} and \ref{7} we obtain: 
\begin{equation}\label{8}
g^nf^m(N_{\e_0} (\atr(g) \setminus \rep(f)) \cap U_f) \subset X.
\end{equation}

Now we consider the set $N_{\e_0} (\atr(g) \setminus \rep(f)) \cap V_f$. We have: $$g^nf^m(N_{\e_0} (\atr(g) \setminus \rep(f)) \cap V_f) \subset g^nf^m(V_f \setminus N_{\e_0}(\rep(f))).$$

Observe that  $f^m(V_f \setminus N_{\e_0}(\rep(f))) \subset N_{\e_1}(\atr(f))$. We have already proved that $g^n(N_{\e_1}(\atr(f))) \subset X$ and so together with Inclusion \ref{8} we have $$g^nf^m(N_{\e_0} (\atr(g) \setminus \rep(f))) \subset X$$ as we wanted.\\

It remains to show that  $g^nf^m(A_{\e_0,\e_1}) \subset X$. Observe that $A_{\e_0,\e_1} \subset V_f \setminus N_{\e_1}(\rep(f))$ and also that $f^{m}(V_f\setminus N_{\e_1}(\rep(f))) \subset N_{\e_1}(\atr(f))$. Using the fact that $g^n(N_{\e_1}(\atr(f))) \subset X$ as we proved before, we obtain: $$g^nf^m(A_{\e_0,\e_1}) \subset g^n(N_{\e_1}(\atr(f))) \subset X.$$

We have shown so far that $g^nf^m(X) \subset X$. Along the way we also proved that $g^nf^m(V_f \setminus N_{\e_0}(\rep(f))) \subset X $. To conclude the proof of Lemma \ref{hard}, we only need to show that $g^nf^m(U_f \setminus N_{\e_0}(\rep(g))) \subset X$.\\

As the set $U_f$ is contained in the set $V_g$ (because $U_f$ and $U_g$ are disjoint) and so  the inclusion $U_f \setminus N_{\e_0}(\rep(g)) \subset W_g$ holds, we obtain:   $$g^nf^m(U_f \setminus N_{\e_0}(\rep(g))) = g^n(U_f \setminus N_{\e_0}(\rep(g))) \subset g^n(W_g)$$ and by Inclusion \ref{4} we have: $$g^n(W_g) \subset N_{\e_0} (\atr(g) \setminus \rep(f)) \cup  A_{\e_0,\e_1} \subset X$$ and so we are done with the proof of Lemma \ref{hard}.
\end{proof}

Now, to finish the proof of Lemma \ref{disjoint} observe that as $X \subset K \setminus N_{\e}(\rep(f) \cup \rep(g))$ and $w_{\epsilon}( K \setminus N_{\e}(\rep(f) \cup \rep(g))) \subset X$, we have that all the periodic points of $w_{\e}$ contained in $ K \setminus N_{\e}(\rep(f) \cup \rep(g))$ are actually contained in $X$, which is a subset  of  $N_{\e}(\atr(f) \cup \atr(g))$. Therefore the periodic points of $w_{\epsilon}$ (namely the set $U_{w_\e} \cup \per(w_\e)$) must be contained in $N_{\e}(\per(f) \cup \per(g))$ as we wanted.
\end{proof}
\end{proof}

As a consequence of Proposition \ref{rf},  for any group $\Gamma \subset V_n$, either there is a finite orbit, a free subgroup, or for every pair of elements $f, g$ in $\Gamma$, we have $U_f \cap U_g \neq \emptyset$. We will generalize Proposition \ref{rf} to an arbitrary number of group elements of $\Gamma$. For any finite set $F \subset \Gamma$, we define the set $K_{F} := \cap_{g \in F} U_g$.

\begin{prop}\label{kf}
Suppose the action of  $\Gamma \subset V_n$ on the Cantor set $K_n$ does not have a finite orbit. Then, for every finite set $F \subset \Gamma$, there exists a finite set $S_{F} \subset K$ such that, for any $\epsilon >0$, there exists an element $h_\e \in \Gamma$ with the following properties:

\begin{enumerate}
\item the set $U_{h_{\e}} \cup \per(h_{\e})$ (the set of periodic points of $h_\e$) is contained in $K_{F} \cup N_{\e}(S_F)$.
\item $h_{\e}$ fixes point-wise $K_F$.

\end{enumerate}

\end{prop}

\begin{proof}

The proof is by induction on the size of the set $F$. The proposition is clearly true if the set $F$ contains only one element. Suppose the result is true for a set $F$ and let $F' = F \cup \{g\}$. By induction hypothesis, we can find  a finite set $S_F$ for $F$ with the desired properties. Consider the set $S_{F'} := S_F \cup \per(g)$. Given $\epsilon > 0$ we can find $h_\e$ that fixes $K_F$ whose periodic points are contained in $K_F  \cup N_{\e/2}(S_F)$.  By taking powers of $h_\e$, we can suppose that $h_\e$ fixes point-wise the clopen set $U_{h_\e}$. We can also suppose that $U_g$ is fixed by $g$ by replacing $g$ with a power $g^k$. Observe that both $h_\e$ and $g$ fix the set $K' := U_{h_\e} \cap U_g$, and so we have two elements $g$ and $h_\e$ preserving the clopen set $C := K \setminus K'$.\\ 

We now consider the actions of $g$ and $h_{\e}$ on our new Cantor set $C$. Restricted to $C$, we have $U_{h_\e} \cap U_g = \emptyset$ and so we are in position to apply Proposition \ref{disjoint} (see Remark \ref{blur} below) to find an element $h'_\e$ in the subgroup $ \langle h_\e,g \rangle \subseteq \Gamma$ such that all the periodic points of $h'_\e$ in $C$  are contained in the set  $ N_{\epsilon/2}(\per(h_\e) \cup \per(g))$. As $\per(h_\e) \subset N_{\e/2}(S_F)$, we obtain that $ \per(h'_\e) \subset N_{\e}(S_{F'})$ and that $U_{h'_{\e}}$ is contained in $(U_{h_\e} \cap U_g) \cup N_{\e}(S_{F'})$, which is a subset of $K_{F'} \cup N_{\e}(S_{F'})$ and so we are done.
\end{proof}

\begin{rem}\label{blur}
Even though Proposition \ref{disjoint} is stated for our original Cantor set $K_n$, it works equally well for actions on clopen sets $C \subset K_n$.

\end{rem}

By applying Lemma \ref{period}, we have the following corollary generalizing Proposition \ref{rf}:

\begin{cor} For every subgroup $\Gamma \subset V_n$ one of the following holds:

\begin{enumerate}
\item The action of $\Gamma$ on $K_n$ has a finite orbit.
\item $\Gamma$ contains a free subgroup on two generators.
\item The set $K_{\Gamma} := \cap_{g \in \Gamma} U_g$ is non-empty.
\end{enumerate}
 
\end{cor}

\begin{proof}

Suppose $\Gamma$ does not have a finite orbit and $K_{\Gamma} = \emptyset$. By the finite intersection property for compact sets, there is a finite set $F \subset \Gamma$ such that $K_{F} = \emptyset$.  By Proposition \ref{kf}, we can find a finite set $S_{F} \subset K$ such that, for every $\e > 0$, there is an element $h_{\e}$ in $\Gamma$ whose periodic points are contained in $N_{\e}(S_F)$. By Lemma \ref{period}, we can find an element $t$ sending $S_{F}$ disjoint from itself and therefore also sending $N_{\e}(S_F)$ disjoint from itself for $\e$ small enough, which implies by Lemma \ref{pingpong} that there is a free group on two generators contained in $\Gamma$ and we are done.

\end{proof}

We will finish the proof of Theorem \ref{main} by proving the following lemma. It is important to point out that it is the only place where we use the finite generation condition on $\Gamma$.

\begin{lem} If $\Gamma$ is a finitely generated subgroup of $V_n$ and $K_{\Gamma} := \cap_{g \in \Gamma} U_g \neq \emptyset$, then the action of $\Gamma$ on $K$ has a finite orbit.

\end{lem}

\begin{proof}

Observe that $K_{\Gamma}$ is $\Gamma$-invariant and therefore there is a minimal closed set $\Lambda \subset K_{\Gamma}$ which is invariant under the action of $\Gamma$ on $K$. If  the minimal set $\Lambda$ is a finite set, then $\Gamma$ has a finite orbit and we are done. We will now show that if $\Lambda$ were infinite, then there would exist an element $g$ in $\Gamma$ with an attracting fixed point in $\Lambda$, contradicting that $\Lambda \subset K_{\Gamma} \subset U_g$.\\

Let $S$ be a finite generating set for $\Gamma$. We can take $\e_0>0$ such that for any $g \in S$ and any elementary interval $I$ of size less than $\e_0$, the element $g$ maps affinely $I$ into the elementary interval $g(I)$. We also take $\e_1 < \e_0$, such that for any elementary interval $I$ of size less than $\e_1$ and $g \in S$, we have that $g(I)$ is an elementary interval of size less than $\e_0$.\\

Take $x \in \Lambda$. For every elementary interval $I_n$ of length less than $\e_1$ containing $x$, we will show arguing by contradiction that there exists $g_n \in \Gamma$ such that $ \e_1 \leq |g_n(I_n)| \leq \e_0$ and such that $g_n|_{I_n}$ is affine. Suppose that there is no such $g_n$ in $\Gamma$. In this case, proceeding by induction on the word length of $g \in \Gamma$, we can see that for every $g \in \Gamma$, $g|_{I_n}$ is an affine map and that $|g(I_n)| < \e_1$.\\

As the orbit $\Gamma (x)$ is dense in $\Lambda$ and as $\Lambda$ is infinite, the point $x$ is a non-trivial accumulation point of $\Gamma(x)$. Hence, for every $I_n$ containing $x$, there exists $h_n \in \Gamma$ such that $h_n(x) \in I_n$ and therefore the intersection $h_n(I_n) \cap I_n$ is non-trivial. Furthermore, we can suppose that $h_n(x) \neq x$. Remember that for any pair of elementary intervals $I_n, I_m$, either one is contained in the other one or they are disjoint.\\

Therefore as the elementary intervals $I_n$ and $h_n(I_n)$ intersect, either $h_n(I_n) = I_n$ or one interval is contained strictly in the other one. If $I_n = h_n(I_n)$, then $h_n|_{I_n} = \text{Id}$ because $h_n$ is affine, contradicting that $h_n(x) \neq x$. We can then suppose that $h_n(I_n)$ is contained strictly into $I_n$. As  the map $h_n|_{I_n}$ is affine, $h_n$ is a contraction inside $I_n$ and therefore $h_n$ has exactly one contracting fixed point $y \in I_n$. This implies that $y$ does not belong to $U_{h_n}$, contradicting the fact that $y =  \lim_{l \to \infty} h_n^l(x) \in \Lambda \subset K_\Gamma$.\\

In conclusion, we found a contradiction for the non-existence of $g_n$ and so for every elementary interval $I_n$ of size less than $\e_1$ containing $x$, there exists $g_n \in \Gamma$ such that  $g_n(I_n)$ is an elementary interval, $g_n|_{I_n}$ is affine and $\e_1 \leq |g_n(I_n)| \leq \e_0$. As there is a finite number of elementary intervals satisfying $\e_1 \leq |I| \leq \e_0$,  there exist two intervals $I_m, I_l$, one contained strictly in the other (Let's say $I_l \subset I_m$) such that $g_m(I_m) = g_l(I_l)$. This implies the element $g := g_{l}^{-1}g_{m}$ is affine on $I_m$ and hence $g|_{I_m}$ is a contraction, which implies that there is a unique contracting fixed point $y$ for $g$ in $I_m$. Hence $y \not\in U_g$, but $y = \lim_{n \to \infty} g^n(x)$ and therefore $y$ also belongs to $\Lambda$. This contradicts the inclusion $\Lambda \subset K_{\Gamma} \subset U_g$. Therefore, we obtain a contradiction to the fact that $\Lambda$ is infinite.

\end{proof}


\end{subsection}

 \end{section}


\begin{thebibliography}{biblio}
 \bibitem{Ali}  Alibegovic, Emina. \emph{"Translation lengths in Out(Fn)"}. Dedicated to John Stallings on the occasion of his 65th birthday. Geom. Dedicata 92 (2002), 87?93.  
 
 \bibitem{B}
 Bleak, Collin, et al. 
 \emph{"Centralizers in R. Thompson's group $V_n$"}
 arXiv preprint arXiv:1107.0672 (2011).
 
 \bibitem{Ba}
 Bavard, Juliette. 
 \emph{"Hyperbolicite du graphe des rayons et quasi-morphismes sur un gros groupe modulaire"}
 arXiv preprint arXiv:1409.6566 (2014).
 
 \bibitem{Br}
 Brin, Matthew G. 
 \emph{"Higher dimensional Thompson groups."}
 Geometriae Dedicata 108.1 (2004): 163-192.
 
 \bibitem{Ca}
 Calegari Danny,
Blogpost: https://lamington.wordpress.com/2014/10/24/mapping-class-groups-the-next-generation/
 
 \bibitem{CFP}
 Cannon, James W., William J. Floyd, and Walter R. Parry. 
 \emph{"Introductory notes on Richard Thompson's groups."}
 Enseignement Math\'ematique 42 (1996): 215-256.
 
\bibitem{Coo} Cooper, Daryl. \emph{"Automorphisms of free groups have finitely generated fixed point sets"}.
J. Algebra 111 (1987), no. 2, 453?456. 
 
 \bibitem{CH1}
 Horbez, Camille. 
 \emph{"A short proof of Handel and Mosher's alternative for subgroups of $\ text {Out}(F_N) $."}
 arXiv preprint arXiv:1404.4626 (2014).
 
 \bibitem{CH2}
 Horbez, Camille. 
 \emph{"The Tits alternative for the automorphism group of a free product."}
 arXiv preprint arXiv:1408.0546 (2014).

\bibitem{Fish} 
Fisher, David. 
\emph{"Groups acting on manifolds: around the Zimmer program."}
arXiv preprint arXiv:0809.4849 (2008).


\bibitem{FH} J. Franks, M. Handel, \emph{Distortion elements in group actions on surfaces}, Duke Math. J. 131 (2006), no 3, 441-468.

\bibitem{FH2} J. Franks, M. Handel, \emph{Periodic points of Hamiltonian surface diffeomorphisms}, Geom. Topol. 7 (2003), 713-756.

 \bibitem{FN}
 Funar, Louis, and Yurii Neretin. 
 \emph{"Diffeomorphisms groups of Cantor sets and Thompson-type groups."}
 arXiv preprint arXiv:1411.4855 (2014).
 
 \bibitem{FLM}
 Farb, Benson, Alexander Lubotzky, and Yair Minsky.
 \emph{ "Rank-1 phenomena for mapping class groups."}
 Duke Mathematical Journal 106.3 (2001): 581-597.
 
\bibitem{FLP} A. Fathi, F. Laudenbach, V. Poenaru, \emph{Travaux de Thurston sur les surfaces}, Astérisque vol. 66, SMF, Paris, France.
 
 
 \bibitem{G}
 Ghys, Etienne.
 \emph{ "Groups acting on the circle."}
  Enseignement Mathematique 47.3/4 (2001): 329-408.
 
  \bibitem{Gro}
  M. Gromov
  \emph{Asymptotic invariants of infinite groups. Geometric group theory. Volume 2}
  Cambridge Univ. Press, Cambridge(1993), 1-295  
  
  \bibitem{H}
  de La Harpe, Pierre. 
  \emph{Topics in geometric group theory.} 
  University of Chicago Press, 2000.
  
  \bibitem{Hir} M. W. Hirsch, 
  \emph{Differential Topology}, 
  Graduate texts in Mathematics, Springer.

\bibitem{HL} Hmili, Hadda; Liousse, Isabelle 
\emph{Dynamique des échanges d'intervalles des groupes de Higman-Thompson Vr,m.}
 Ann. Inst. Fourier (Grenoble) 64 (2014), no. 4, 1477-1491

\bibitem{Mac} McCarthy, John.
\emph{A "Tits-alternative'' for subgroups of surface mapping class groups.} Trans. Amer. Math. Soc. 291 (1985), no. 2, 583-612.

\bibitem{M}
Margulis, Gregory. 
\emph{"Free subgroups of the homeomorphism group of the circle."}
 Comptes Rendus de l'AcadŽmie des Sciences-Series I-Mathematics 331.9 (2000): 669-674.
 
 \bibitem{N}
 Navas, Andres.
 \emph{Groups of circle diffeomorphisms.} 
 University of Chicago Press, 2011.
 
 \bibitem{SD}
Salazar-Diaz, Olga Patricia.
\emph{"Thompsons group $V$ from a dynamical viewpoint."}
International Journal of Algebra and Computation 20.01 (2010): 39-70.
APA	

\bibitem{Z}
Robert J. Zimmer,
\emph{Ergodic theory and semisimple groups}
  

\end{thebibliography}
\end{document}